\documentclass[11pt,oneside,reqno]{amsart}
\usepackage{tikz}
\definecolor{red2}{RGB}{220,10,48}
\usepackage[english]{babel}
\usepackage[all,cmtip]{xy}
\usepackage{lmodern}
\usepackage{amsmath,amsfonts,amsthm,amssymb,amsbsy,upref,color,graphicx,amscd,hyperref,enumerate,bbm,comment, algorithm, algorithmic,dsfont}
\hypersetup{colorlinks,breaklinks,
            linkcolor=[rgb]{0,0,0},
            citecolor=[rgb]{0,0,0},
            urlcolor=[rgb]{0,0,0}}
\usepackage{xfrac,xcolor}
\usepackage{mathtools,url}
\usepackage[active]{srcltx}
\usepackage{cases} 
\usepackage[shortlabels]{enumitem}
\usepackage{cancel}
\usepackage[capitalise, nameinlink, noabbrev]{cleveref} 

\newcommand{\norm}[2]{\left\Vert {#1} \right\Vert_{#2}}

\newcommand{\mathbbmm}[1]{\text{\usefont{U}{bbm}{m}{n}#1}}
\newcommand{\ind}{\mathbbmm{1}}

\DeclareMathOperator{\dive}{div}

\newcommand{\mres}{\mathbin{\vrule height 1.6ex depth 0pt width
0.13ex\vrule height 0.13ex depth 0pt width 1.3ex}}

\def\eps{{\varepsilon}}
\def\N{\mathbb{N}}
\def\O{\Omega}
\def\R{\mathbb{R}}

\def\HH{\mathcal{H}}

\newcommand{\be}{\begin{equation}}
\newcommand{\ee}{\end{equation}}

\newcommand{\loc}{\mathrm{loc}}

\numberwithin{equation}{section}
\theoremstyle{definition}
\newtheorem{theo}{Theorem}[section]
\newtheorem{theorem}[theo]{Theorem}

\newtheorem{lemma}[theo]{Lemma}

\newtheorem{definition}[theo]{Definition}

\newtheorem{open}[theo]{Open problem}
\theoremstyle{remark}
\newtheorem{remark}[theo]{Remark}

\def\XXint#1#2#3{{\setbox0=\hbox{$#1{#2#3}{\int}$ }
\vcenter{\hbox{$#2#3$ }}\kern-.6\wd0}}

\title{Vectorial Bernoulli Problems and Free Boundary Systems}
\author[G.~Tortone, B.~Velichkov]{Giorgio Tortone and Bozhidar Velichkov}

\address {Giorgio Tortone \newline \indent
	Dipartimento di Matematica ``Giuseppe Peano'', Universit\`a di Torino \newline \indent
Via Carlo Alberto 10, 10123 Torino, Italy}
\email{giorgio.tortone@unito.it}

\address {Bozhidar Velichkov \newline \indent
Dipartimento di Matematica, Universit\`a di Pisa \newline \indent
Largo Bruno Pontecorvo, 5, 56127 Pisa, Italy}
\email{bozhidar.velichkov@unipi.it}

\begin{document}


\begin{abstract}
In this survey we go through some of the recent results about the regularity of vectorial free boundary problems of Bernoulli type and free boundary systems. The aim is to illustrate the general methodologies as well as to outline a selection of notable open questions.
\end{abstract}

\maketitle
\setcounter{tocdepth}{1}
\tableofcontents
\section{Introduction}
Free boundary problems are a special class of boundary value problems in which the domain where the PDE is solved is not fixed in advance, but depends on the solution of the PDE. These problems involve solving simultaneously for the unknown domain and the unknown function by exploiting the presence of an overdetermined boundary condition on the "free" part of the boundary of the domain,
which is usually dictated by physical laws or other constraints governing the phase transition.

This survey is dedicated to the Bernoulli free boundary systems, which are free boundary problems in which the unknown function is vector-valued and the overdetermined boundary condition is of Bernoulli type. We will briefly explain the motivations for introducing this class of problems, we will highlight some of the key techniques and recent results in the study of vectorial free boundaries, and finally, we will outline a selection of open questions.

\subsection{The one-phase Bernoulli problem} We start with a brief overview of the theory of minimizers for the one-phase Bernoulli problem, which has played a central role in shaping much of the recent research in the field of free boundary problems. 
The one-phase problem can be described as follows: given $\Lambda>0$, a bounded open set $D\subset \R^d$ (container), a function  $g:\partial D\to \R$ (the boundary condition), one wants to find a non-negative function $u\colon D \to \R$ such that 
\be\label{e:one-phase-bernoulli}
\Delta u = 0 \quad\mbox{in }\{u>0\},\qquad 
u=g \quad\mbox{on }\partial D
\qquad \mbox{and}\qquad 
|\nabla u|=\sqrt{\Lambda}\quad\text{on}\quad D\cap \partial\{u>0\}\,.
\ee
The free boundary is  defined as the portion of  $\partial\{u>0\}$ lying inside the open set $D$. 
The main peculiarity of Bernoulli-type free boundary problems is the presence of an overdetermined condition, where the behavior of the gradient of the solution is prescribed on the free interface, which implies a discontinuity at the level of the partial derivatives.

\subsubsection{Physical motivation} The one-phase Bernoulli problem arises naturally in models in fluid dynamics; here we briefly present a classical model and we refer to the book  \cite{Friedman1982:BookVariationalPrinciplesFreeBoundaryProblems} for more details and further discussion on other related models. We consider a fluid flowing through a (simply connected open) region $\Omega\subset\R^2$. We suppose that the fluid is incompressible and irrotational, so that its velocity $v=(v_1,v_2):\overline \Omega\to\R^2$ satisfies the PDEs 
$$\text{\rm div\,}v=0\quad\text{and}\quad \text{\rm rot\,}v=0\quad\text{in}\quad \Omega\,.$$
Finally, we suppose that there is a relatively open and connected portion of the boundary $\Gamma\subset\partial \Omega$ such that:
\begin{itemize}
\item[(i)] the velocity $v$ is prescribed on $\partial \Omega\setminus\Gamma$;
\item[(ii)] $\Gamma$ is a streamline of $v$ (that is, $v$ is always tangent to $\Gamma$);
\item[(iii)] the pressure is constant on $\Gamma$;
\item[(iv)] the velocity $v$ satisfies the Bernoulli law 
$$| v|^2=\Lambda\quad\text{on}\quad \Gamma,$$
for a non-zero constant $\Lambda$.
\end{itemize}
Then, one can find a function $u:\overline \Omega\to\R$ called {\it stream function} such that: 
$$\Delta u=0\quad\text{in}\quad \Omega\ ,\qquad \nabla u=(-v_2,v_1)\quad\text{in}\quad \overline \Omega\ ,\quad u=0\quad\text{on}\quad \Gamma.$$
In terms of the stream function, the Bernoulli law translates into the overdetermined condition 
$$|\nabla u|=\sqrt{\Lambda}\quad\text{on}\quad\Gamma,$$
so the stream function $u$ satisfies \eqref{e:one-phase-bernoulli} in a neighborhood of every boundary point $x_0\in\Gamma$. Vice-versa, if $u$ solves \eqref{e:one-phase-bernoulli}, then $(-\partial_2u,\partial_1u)$ is the velocity field of a fluid satisfying the Bernoulli law on $\Gamma=\Omega\cap\partial\{u>0\}$.

\subsubsection{Variational formulation} The existence of solutions $u$ satisfying \eqref{e:one-phase-bernoulli} can be obtained via a variational principle. Precisely, in the pioneering paper \cite{AltCaffarelli:OnePhaseFreeBd} Alt and Caffarelli developed a regularity theory for minimizers $u:D\to\R$ of the functional 
\be\label{e:ac-functional}
\mathcal F(v,D) :=\int_D|\nabla v|^2\,dx+\Lambda|\{v>0\}\cap D|,
\ee
among all functions $v\in H^1(D)$ with fixed boundary values $g$ on $\partial D$ (that is, $v-g\in H^1_0(D)$).\\
It is immediate to check that if the boundary datum $g:\partial D\to\R$ is nonnegative, then any minimizer $u\in H^1(D)$ of $\mathcal F$ is nonnegative, so the sign of $u$ is not a constraint (as for instance in the obstacle problem), but a simple consequence of the variational structure of the problem. A key feature of the functional \eqref{e:ac-functional} is that the overdetermined boundary condition \eqref{e:one-phase-bernoulli} can be obtained as a consequence of the minimality of $u$. 
Precisely, let us denote with $\O_u:=\{u>0\}\subset D$ the positivity set of $u$, then the equation satisfied by $u$ inside the set $\O_u\subset D$ and the overdetermined condition on the free boundary $\partial \O_u\cap D$ are in fact the Euler-Lagrange equations associated with two specific types of variations:
\begin{enumerate}
    \item[(i)] \textbf{Outer variations. } The minimality of $u$ gives that
     \be\label{e:outer0}
     \mathcal{F}\left(u(x),D\right)\le \mathcal{F}\left(u(x)+ \varphi(x),D\right) \quad\mbox{for every }\varphi \in C^\infty_c(D).
    \ee  
    By choosing appropriately the perturbation $\varphi$ one can prove that $u$ is Lipschitz continuous in $D$ and has linear growth around points on $\partial\Omega_u\cap D$.
    In particular, $\Omega_u$ is open and by computing the first (outer) variation for functions supported in $\Omega_u$
    \be\label{e:outer}
    \left.\frac{d}{dt}\right\lvert_{t=0}\mathcal{F}\left(u(x)+t \varphi(x),D\right) =0\quad\mbox{for every }\varphi \in C^\infty_c(\O_u),
    \ee
    one gets simply
    \be\label{e:outer2}
    \int_{\Omega_u}\nabla u\cdot\nabla\varphi\,dx=0 \quad\mbox{for every }\varphi \in C^\infty_c(\O_u),
    \ee
    so that $u$ is harmonic in $\O_u$.\vspace{0.2cm}
    \item[(ii)] \textbf{Inner variations. } For every smooth vector field $\xi:D\to \R^d$ with compact support we have that $$
     \mathcal{F}\left(u(x),D\right)\le \mathcal{F}\left(u(x+t\xi(x)),D\right),
    $$
    so we can compute the first (inner) variation with respect to vector fields supported in $D$ 
    \be\label{e:inner}
    \left.\frac{d}{dt}\right\lvert_{t=0}\mathcal{F}\left(u(x+t\xi(x)),D\right) =0\qquad\mbox{for every}\quad\xi \in C^\infty_c(D;\R^d),
    \ee
    which can be written in the form 
        \be\label{e:inner2}
\int_{\Omega_u}\text{\rm div}\,\Big(|\nabla u|^2\xi-2(\xi\cdot\nabla u)\nabla u+\Lambda\xi\Big)\,dx=0\qquad\mbox{for every}\quad\xi \in C^\infty_c(D;\R^d).
\ee
This identity allows to prove the Weiss' monotonicity formula, which is fundamental for the analysis of the blow-up limits and the singular set. Moreover, if the free boundary $\partial\O_u\cap D$ is $C^{1,\alpha}$-regular (so that $u \in C^{1,\alpha}(\overline{\O_u}\cap D)$ by classical elliptic regularity), then one can deduce the validity of the free boundary condition 
\begin{equation}\label{e:stationarity-condition-for-smooth-boundaries}
|\nabla u|=\sqrt{\Lambda}\quad\text{on}\quad\partial\O_u\cap D,
\end{equation}
simply by integrating by parts the above expression (see for instance \cite[Lemma 9.5]{Velichkov:BookRegularityOnePhaseFreeBd}). 
    \end{enumerate}
If $u:D\to\R$ is a non-negative function with smooth free boundary $\partial\{u>0\}\cap D$, which is harmonic in $\{u>0\}\cap D$, we will say that {\it $u$ is a solution of the one-phase Bernoulli problem} or a {\it stationary point of the Alt-Caffarelli energy} if \eqref{e:stationarity-condition-for-smooth-boundaries} holds; we refer to \cite{liu-wang-wei,hauswirth-helein-pacard,traizet} for examples of non-trivial smooth stationary solutions.  We also notice that the notion of stationary solution can be defined without assuming the regularity of $u$ and $\partial\{u>0\}$, but only by using the expressions \eqref{e:inner} and \eqref{e:outer0} for the outer and the inner variations. 

\subsubsection{Regularity theory for minimizers} We summarize the main regularity results for one-phase minimizers $u$ and their positivity sets $\Omega_u$ in the next theorem; for more details we refer to \cite{Velichkov:BookRegularityOnePhaseFreeBd}. 
\begin{theorem}\label{t:one-phase-scalar}
    Let $u\colon D \to \R$ be a minimizer of $\mathcal{F}$ in $D$. Then $u$ is (locally) Lipschitz continuous in $D$ and the set $\O_u := \{u>0\}$ is open. Moreover, the free boundary can be decomposed as the disjoint union of a regular and a singular part 
    $$
    \partial \O_u \cap D = \mathrm{Reg}(\partial \O_u) \cup \mathrm{Sing}(\partial \O_u),
    $$
    where:\vspace{-0.8cm}\\
     \begin{minipage}{1.5\linewidth}
\begin{minipage}{0.38\linewidth}
\vspace{0.1cm}
\begin{enumerate}\vspace{0.1cm}
        \item[(i)] the \textbf{Regular part} $\mathrm{Reg}(\partial \O_u)$ is a smooth manifold of dimension $(d-1)$ and $u$ satisfies the free boundary condition
        $$
        |\nabla u|=\sqrt{\Lambda} \quad \mbox{on }\mathrm{Reg}(\partial \O_u);\vspace{0.2cm} 
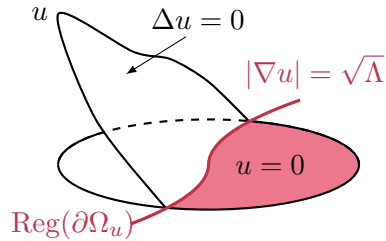
        $$
\end{enumerate}
  \end{minipage}
\hspace{0.01\linewidth}
 \begin{minipage}{0.3\linewidth}
                      \vspace{0.5cm}\begin{tikzpicture}[scale=2.0]
 \def\a{-0.53}
 \def\b{0.52}
 \def\p{0.55}
 \def\q{1.1}
  \begin{scope}
  \clip (0,0) ellipse (1cm and 0.3cm);
\filldraw[red2!55, domain=-0.8:0.78]  (0,-1) -- plot (\x^2,0.55*\x ) -- (1.5,0) -- (0,-1) -- cycle;
 \end{scope}
    \draw[black] (0.41,0) node {$u=0$};
       \draw[black,thick] (1,0) arc (0:180:1cm and 0.3cm);
  \draw[black,thick] (-1,0) arc (180:360:1cm and 0.3cm);
   \fill [white] plot [smooth] coordinates {(\a^2,0.55*\a) (-0.75,0.3)  (-1,1) (-0.5,0.75) (-0.25,0.7) (0,0.55) (\b^2,0.55*\b) (0.5^2,0.55*0.5)
    (0.27^2,0.55*0.3) (0,0) (-0.29^2,-0.55*0.3) (-0.5^2,-0.55*0.5) (\a^2,0.55*\a) };
   \draw [black,thick] plot [smooth] coordinates {(\a^2,0.55*\a) (-0.75,0.3)  (-1,1) (-0.5,0.75) (-0.25,0.7) (0,0.55) (\b^2,0.55*\b) };
      \draw[color=red2!65!gray, very thick, domain=-0.72:0.78] plot (\x^2,0.55*\x ) node[anchor  =south ]{$\quad|\nabla u|=\sqrt{\Lambda}$};
\draw[color=red2!65!gray] (-0.9,-0.4) node {$\mathrm{Reg}(\partial \O_u)$};
          \draw[black,thick,dashed] (1,0) arc (0:180:1cm and 0.3cm);
   \draw (-1,1) node[anchor=east] {$u$};
   \draw [black] (-0.8,1.2) circle (0pt);
   \draw[black]  (-0.45,0.97) node[anchor=west] {$\Delta u=0$};
   \draw[black,latex-]  (-0.55,0.6) --(-0.15,0.85) ;
   \end{tikzpicture}
      \end{minipage}
  \end{minipage}
    \begin{enumerate}\vspace{0.1cm}
        \item[(ii)] the \textbf{Singular part} $\mathrm{Sing}(\partial \O_u)$ is a closed subset of $\partial \O_u \cap D$. Moreover\vspace{0.2cm}  
        \begin{itemize}
\item[(1)] if $d<d^*$, then $\text{\rm Sing}(\partial \Omega_u)$ is empty;
\item[(2)] if $d=d^*$, then $\text{\rm Sing}(\partial \Omega_u)$ is a discrete set of points;
\item[(3)] if $d> d^*$, then the Hausdorff dimension of $\text{\rm Sing}(\partial \Omega_u)$ is at most $d-d^*$,\vspace{0.2cm}
\end{itemize}
        where $d^* \in \{5,6,7\}$ is the smallest dimension in which
1-homogeneous minimizers of $\mathcal F$ in $B_1$ may exhibit isolated singularities.
    \end{enumerate}
\end{theorem}
The overall approach and many of the used tools are universal and
have counterparts in other geometric variational problems. For the sake of completeness we just mention the different contributions for the main regularity result:
\begin{enumerate}
    \item[(i)] In \cite{AltCaffarelli:OnePhaseFreeBd}, Alt and Caffarelli proved that minimizers exist and that they are (locally) Lipschitz continuous. Moreover, they decompose the free boundary in terms of a $C^{1,\alpha}$-regular $(d-1)$-dimensional surface and a singular set satisfying $\mathcal{H}^{d-1}(\mathrm{Sing}(\partial \O_u))=0$.\smallskip
    \item[(ii)] \textbf{Epsilon-regularity theorem. }In \cite{DeSilva:FreeBdRegularityOnePhase}, De Silva introduced a more versatile approach for the $\eps$-regularity theorem of Alt and Caffarelli, which exploits a viscosity formulation of the free boundary condition. Finally, thanks to \cite{KinderlehrerNirenberg1977:AnalyticFreeBd,KinderlehrerNirenbergSpruck} and the appendix in \cite{KriventsovLin2018:nondegenerate}, we know that the $C^{1,\alpha}$ free boundaries are $C^\infty$ smooth.\smallskip
    \item[(iii)] \textbf{Monotonicity formula. }In \cite{Weiss99:PartialRegularityFreeBd}, Weiss proved the validity of a monotonicity formula, which allows to estimate the Hausdorff dimension of the singular set in terms of the first dimension $d^*$ in which homogeneous minimizers exhibit isolated singularities.
    \end{enumerate}
    \vspace{-0.3cm}
    \begin{minipage}{1.5\linewidth}
    \begin{minipage}{0.37\linewidth}
      \begin{enumerate}
      \item[(iv)] \textbf{Singularities and stability.} In \cite{CaffarelliJerisonKenig04:NoSingularCones3D} Caffarelli, Jerison and Kenig computed the second variation of $\mathcal{F}$ along inner variations, which allows to exclude the presence of singularities in three dimensions (i.e., $d^*\geq 4$). 
    Few years later, in \cite{DeSilvaJerison09:SingularConesIn7D} De Silva and Jerison construct in $\R^7$ an homogeneous minimizer exhibiting an isolated singularity (i.e., $d^*\leq 7$) while in \cite{JerisonSavin15:NoSingularCones4D} Jerison and Savin improved the lower bound of $d^*$ by showing that $d^*\geq 5$.
     \end{enumerate}
      \end{minipage}\hspace{0.05\linewidth}
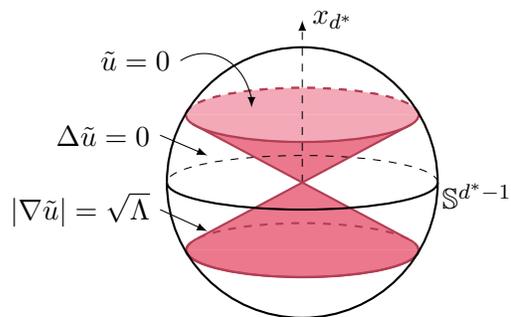
\begin{minipage}{0.37\linewidth}
\hspace{-1.3cm}
      \begin{tikzpicture}[scale=1.8]
   \begin{scope}
  \clip (0,0) ellipse (1 and 1);
\fill[white] (0,0)-- (1,0.6) -- (1,-0.6) -- (0,0);
 \end{scope}
   \begin{scope}
  \clip (0,0) ellipse (1 and 1);
\fill[white] (0,0)-- (-1,0.6) -- (-1,-0.6) -- (0,0);
 \end{scope}
   \begin{scope}
  \clip (0,0) ellipse (1 and 1);
\fill[red2!55] (0,0)-- (0.805,-0.43) -- (-0.805,-0.43) -- (0,0);
 \end{scope}
 \begin{scope}
  \clip (0,0) ellipse (1 and 1);
\fill[red2!55] (0,0)-- (0.805,0.43) -- (-0.805,0.43) -- (0,0);
 \end{scope}
 \draw[thick,red2!55!gray,dashed,fill=white] (0.859,0.5) arc (0:180:0.859cm and 0.2cm);
  \draw[thick,red2!55!gray,fill=white] (-0.859,0.5) arc (180:360:0.859cm and 0.2cm);
  \draw[thick,red2!55!gray,dashed,fill=red2!35] (0.859,0.5) arc (0:180:0.859cm and 0.2cm);
  \draw[thick,red2!55!gray,fill=red2!35] (-0.859,0.5) arc (180:360:0.859cm and 0.2cm);

  \draw [black,-latex, dashed] (0,0) -- (0,1.2) node [anchor = west] {$x_{d^*}$};
  \filldraw [black] (1.29,-0.1) circle (0pt) node {$\mathbb{S}^{d^*-1}$}; 
\draw [black] (-0.75,1.6) circle (0pt);

   \draw[black,dashed] (1,0) arc (0:180:1cm and 0.2cm);
   \draw[thick,red2!55!gray,dashed,fill=red2!55] (0.859,-0.5) arc (0:180:0.859cm and 0.2cm);
  \draw[thick,red2!55!gray,fill=red2!55] (-0.859,-0.5) arc (180:360:0.859cm and 0.2cm);
       \draw[thick,red2!55!gray] (-0.803,-0.43) arc (160:381:0.856cm and 0.2cm);
     \draw [thick,red2!55!gray] (0.805,0.43)  -- (-0.805,-0.43);
   \draw [thick,red2!55!gray] (-0.805,0.43) -- (0.805,-0.43);
   \draw[thick,black] (1,0) arc (0:180:1cm);
     \draw[thick,black] (-1,0) arc (180:360:1cm and 0.2cm);
   \draw[thick, black] (-1,0) arc (-180:0:1cm);
     \draw[black,-latex] (-0.9,0.9) node[anchor=east] {$\tilde{u}=0$} arc (120:0:0.35cm and 0.4cm);
      \draw[black] (-1.4,-1.3) circle (0pt) ;
   \draw[black,-latex]  (-1.05,0.35) node[anchor=east] {$\Delta \tilde{u}=0$} -- (-0.7,0.2);
   \draw[black,-latex]  (-1.05,-0.2) -- (-0.7,-0.36);
\draw (-1.05,-0.2) node[anchor=east] {$|\nabla \tilde{u}|=\sqrt{\Lambda}$};
        \end{tikzpicture}
        \end{minipage}    
  \end{minipage} 
Finally, we mention that the classical one-phase Bernoulli problem has inspired numerous recent developments in the theory of free boundaries and in different problems arising from Calculus of Variations and PDEs. For reader convenience, we mention here few examples in the case of two-phases problem \cite{AltCaffarelliFriedman1984:TwoPhaseBernoulli,DePhilippisSpolaorVelichkov:Inventiones,DeSilvaFerrariSalsa:2PhaseFreeBdDivergenceForm}; shape optimization problems \cite{RussTreyVelichkov:CalcVarPDE, BrianconLamboley:RegularityShapeFirstEigenLaplacian,Briancon2010}; thin free boundaries \cite{DS1,DS2,DS3,EKPSS}; 
nonlocal operators \cite{RosOtonWeidner:improvement,RosOtonWeidner:optimal}; 
capillarity problems \cite{ChodoshEdelenLi2024,DeMasiEdelenGasparettoLi,DePhilippisFuscoMorini:drop,FerreriTortoneVelichkov:Capillarity}. Finally, for what concerns the up-to-the-fixed-boundary regularity of the solutions and their free boundaries, we refer to \cite{ChangLaraSavin:BoundaryRegularityOnePhase,DePhilippisSpolaorVelichkov:JEMS,FerreriSpolaorVelichkov:2D,FerreriSpolaorVelichkov:BoundaryBranchingOnePhaseBernoulli}.
\subsection{General Bernoulli free boundary systems}\label{sub:intro:general-systems-short} The Bernoulli free boundary systems are generalizations of the one-phase problem \eqref{e:one-phase-bernoulli} to the case of vector-valued functions. 
A very general form of a Bernoulli free boundary system is the following. Given $k\in \N$, a bounded open set $D\subset \R^d$, and a boundary datum 
$$G:=(g_1,\dots,g_k):\partial D\to\R^k,$$
we search for a domain $\O\subset D$ and a vector-valued function $U=(u_1,\dots,u_k)\colon \O \to \R^k$ such that
\begin{enumerate}
    \item[(i)] $U$ solves a system of PDEs of the form
    $$
    \sum_{j=1}^k\mathrm{tr}(A_{ij}\nabla^2 u_j) + (\nabla u_j\cdot b_{ij}) + c_{ij} u_j = f_i\quad \mbox{in }\O;
    $$
    \item[(ii)] $U$ is prescribed on the fixed part of the boundary, that is 
    $U=G$ on $\partial D\cap\partial\Omega$;
    \item[(iii)] $U$ satisfies a Bernoulli-type condition on the free boundary $\partial \O\cap D$
    \be\label{e:vectorial-bernoulli-general}
    U=0\quad\text{and}\quad \Phi(x,\nabla u_1,\dots,\nabla u_k)=0\quad\text{on}\quad D\cap \partial\O,
    \ee
    for some function $\Phi\colon D \times (\R^d)^k \to \R$.
\end{enumerate}
The focus here is not on the second order system the interior of $\Omega$, but on the free boundary condition, which is not given in terms of the graph of a single function, but is obtained as a combination of the gradients of the different components of $U$. It is precisely this interaction at the free boundary that makes the regularity theory for this problem challenging.  

The theory of Bernoulli free boundary systems was developed in the last years thanks to contributions of several different groups of authors. In the next two subsections we discuss the Bernoulli free boundary systems that have been studied in these years.

\subsection{The vectorial Bernoulli problem}\label{sub:intro:vectorial} Let the domain $D\subset\R^d$ and the dimension $k\ge 1$ be fixed. For a vector-valued function $U\colon D\to \R^k$, we use the notations
$$
|U|:=\sqrt{u_1^2 + \cdots + u_k^2},\quad \nabla U = (\nabla u_1,\dots,\nabla u_k)\in (\R^d)^k,\quad |\nabla U|^2 = \sum_{i=1}^k |\nabla u_i|^2
$$
and
$$
\O_U := \{|U|>0\} = \bigcup_{i=1}^k\{u_i\neq 0\}.
$$ 
Consider the function $\Phi(x,p):=|p|^2 - \Lambda$ in \eqref{e:vectorial-bernoulli-general} with $\Lambda>0$. With this choice of $\Phi$, the system \eqref{e:vectorial-bernoulli-general} corresponds to the Euler-Lagrange system associated to the following vectorial version of the one-phase Alt-Caffarelli's functional
\be\label{e:vectorial-bernoulli-functional}
\mathcal J(V,D) :=\int_D|\nabla V|^2\,dx+\Lambda|\{|V|>0\}\cap D|,
\ee
defined for vector-valued functions $V\in H^1(D;\R^k)$. 

Let $U:=(u_1,\dots,u_k)$ be a minimizer of \eqref{e:vectorial-bernoulli-functional} among all functions $V\in H^1_0(D;\R^k)$ with prescribed boundary values $V=G$ on $\partial D$.
By performing outer and inner variations, one can show that $U$ solves the following \it vectorial Bernoulli problem \rm 
\be\label{e:vectorial-bernoulli}
\Delta U = 0 \quad\mbox{in }\O_U,\qquad 
\mbox{and}\qquad 
|\nabla U|^2={\Lambda}\quad\text{on }D\cap \partial \O_U\,.
\ee

\subsubsection{History and motivation} The variational vectorial problem \eqref{e:vectorial-bernoulli} was introduced in the papers \cite{CaffarelliShahgholianYeressian2018:vectorial, MazzoleniTerraciniVelichkov:GAFA, KriventsovLin2018:nondegenerate} which appeared simultaneously in 2016. In \cite{CaffarelliShahgholianYeressian2018:vectorial} the functional \eqref{e:vectorial-bernoulli-functional} is motivated by a thermal insulation model, in which the heat loss from the insulating layer is prescribed through a vector-valued boundary datum $G=(g_1,\dots,g_k)$. Each component $g_i$ is interpreted as an independent source of temperature that gives rise to a potential $u_i\colon D \to \R$.
The cost functional is the sum of the Dirichlet energy of $U := (u_1, \dots, u_k)$ and the volume of the heated region $\{|U| > 0\}$. Different distributions of the source $G$ lead to different optimal configurations of $U$: for instance, if the functions $g_i$ are small and with distant supports, it is reasonable to expect the system to behave like a collection of independent scalar problems; conversely, if the components $g_i$ are large and if their supports overlap, then it becomes energetically more favorable for the functions $u_i$ to share a common support.
The focus of \cite{CaffarelliShahgholianYeressian2018:vectorial} is on the case in which all components of the boundary datum $G$ and of the minimizer $U$ are non-negative, i.e.
\be\label{e:ass.CSY}
g_i\geq 0,\, u_i\geq 0 \quad\mbox{a.e. in }D,\mbox{ for every }i=1,\dots,k.
\ee
In \cite{KriventsovLin2018:nondegenerate} and \cite{MazzoleniTerraciniVelichkov:GAFA} the free boundary regularity of the minimizers of the functional \eqref{e:vectorial-bernoulli-functional} was studied in connection with the following variational problem
\be\label{e:spectral-opt-prob}
\mathrm{min}\left\{F(\lambda_1(A),\dots,\lambda_k(A)): A\subset \R^d \mbox{ quasi-open},\, |A|=c\right\}
\ee
in which the variables are the open subsets of $\R^d$. Here the constant $c>0$ is given, the function  $F:=F(t_1,\dots,t_k)$ is non-negative and $C^1$, and $\lambda_i(A)$ is the $i$-th eigenvalue of the Dirichlet Laplacian on $A$. In the simplest case 
    \be\label{e:somma-intro}
    F(A):=\lambda_1(A) + \cdots + \lambda_k(A),
    \ee 
   the regularity of the optimal sets was studied in \cite{MazzoleniTerraciniVelichkov:GAFA}, while \cite{KriventsovLin2018:nondegenerate} was dedicated to the more general case of  functionals $F$ satisfying the following {non-degeneracy condition}: 
    $$\partial_{\lambda_i}F\geq \mu>0\qquad\text{for every}\qquad i=1,\dots,k.$$
    In both cases (\cite{KriventsovLin2018:nondegenerate} and \cite{MazzoleniTerraciniVelichkov:GAFA}) the problem of finding a minimizer for \eqref{e:spectral-opt-prob} among the open sets $A$ can be written as a variational problem for the eigenfunctions $(v_1,\dots,v_k)$ on $A$. For instance, in the case of the functional \eqref{e:somma-intro} the free boundary system satisfied by a minimizer $U=(u_1,\dots,u_k)$  is the following:  
    \be\label{e:vectorial-bernoulli-sum-lambda-k-equation}
-\Delta u_j = \lambda_j(\O)u_j \quad\mbox{in }\O,\qquad 
\mbox{and}\qquad 
|\nabla U|^2=\frac{2}{d}\sum_{i=1}^k\lambda_i(\O)\quad\text{on }D\cap \partial \O\,.
\ee
A key feature in the papers \cite{CaffarelliShahgholianYeressian2018:vectorial,KriventsovLin2018:nondegenerate,MazzoleniTerraciniVelichkov:GAFA} is the existence of at least one components of $U$, which is non-negative; in \cite{KriventsovLin2018:nondegenerate} and \cite{MazzoleniTerraciniVelichkov:GAFA}, for instance, we have 
$$\{u_1>0\}=\{|U|>0\}=\Omega_U.$$
Roughly speaking, this allows to re-write the problem as a one-phase free boundary problem for $u_1$ and to show that $\partial\Omega_U$ is regular, up to a closed singular set of Hausdorff dimension at most $d-5$, by exploiting the regularity theory for the one-phase problem.

\subsubsection{The degenerate case} Let $U=(u_1,\dots,u_k):D\to\R^k$ be a minimizer of the vectorial problem \eqref{e:vectorial-bernoulli-functional} with boundary datum $G=(g_1,\dots,g_k):\partial D\to\R^k$ and suppose that we are in the so-called {\it degenerate case} in which all the components of $G$ change sign on $\partial D$. Then, all the components of the solution $U$ have to change sign in $D$, so the theory from \cite{CaffarelliShahgholianYeressian2018:vectorial, KriventsovLin2018:nondegenerate,MazzoleniTerraciniVelichkov:GAFA} does not apply. We notice that this is not a mere technical issue. In fact, in the degenerate case, the free boundary $\partial\Omega_U\cap D$ can develop cusps (boundary points at which the Lebesgue density of $\R^d\setminus\Omega_U$ is zero) even in dimension $d=2$. Moreover, in any dimension $d\ge 2$, the set of the cusp-like singularities can reach dimension $d-1$, which is the same dimension that the regular part has. \medskip

\noindent{\it Degenerate solutions and shape optimization problems.} Degenerate vectorial problems arise from shape optimization problems as \eqref{e:spectral-opt-prob}, in which the functional depends only on higher eigenvalues. Also in this case cusp-like singularities might appear on the boundaries of the optimal sets. Following \cite{KriventsovLin2019:degenerate}, we call a functional $F(\lambda_1,\dots,\lambda_k)$ {\it degenerate} if the function $F$ is increasing in each variable, but its partial derivatives are not bounded from below by positive constants:
$$\partial_{\lambda_i}F\ge0\qquad\text{for every}\qquad i=1,\dots,k.$$
A typical example of a degenerate functional is 
\begin{equation}\label{e:k-intro}
F(\lambda_1,\dots,\lambda_k)=\lambda_k.
\end{equation}
Kriventsov and Lin showed in \cite{KriventsovLin2019:degenerate} that if $\Omega$ is an optimal set for $\lambda_k$, then there exists a family of indices $\mathcal{I}\subset\{1,\dots,k\}$ and of normalized eigenfunctions $(u_i)_{i\in \mathcal{I}}$ such that  
$$
\sum_{i\in\mathcal{I}} \xi_i |\nabla u_i|^2 = \frac{2}{d}\lambda_k(\O)\qquad\mbox{on}\quad\partial \O \cap D,
$$
where $(\xi_i)_{i\in \mathcal{I}}$ are positive coefficients. Since the set $\mathcal I$ might not contain the index $1$, the eigenfunctions $\{u_i\}_{i\in\mathcal I}$ might change sign on $\Omega$, so the vector of eigenfunctions $(u_i)_{i\in\mathcal I}$ is a (viscosity) solution of a degenerate vectorial problem. Finally, we notice that the degenerate vectorial problems and the shape optimization problems for degenerate functionals $F$ share the same technical difficulties and the free boundaries of their solutions have similar structure.\medskip

\noindent{\it Epiperimetric inequality and regularity of the free boundaries in 2D.} The free boundary regularity for solutions to the degenerate vectorial problem was first approached in \cite{SpolaorVelichkov:CPAM} (which appeared in 2016) via an epiperimetric inequality in dimension $d=2$. This epiperimetric inequality allowed to prove that, for $d=2$ and for any $k\ge 1$, the free boundary $\partial\Omega_U\cap D$ can be decomposed as a regular part $Reg$ and a singular part $Sing$ such that: 
\begin{itemize}
\item $Reg$ consists only of boundary points at which the Lebesgue density of $\Omega_U$ is $1/2$; $Reg$ is relatively open subset of $\partial\Omega_U$ and is locally the boundary of a smooth (analytic) set; \smallskip
\item the singular part  $Sing$ consists only of boundary points at which the Lebesgue density of $\Omega_U$ is $1$; moreover, $Sing$ is  locally contained in $C^{1,\alpha}$ regular curves.\medskip
\end{itemize}

\noindent{\it Epsilon-regularity results in higher dimension.} The smoothness of the flat free boundaries in the degenerate case, in general dimension $d\ge 2$ and for solutions of the shape optimization problem \eqref{e:spectral-opt-prob}, was first obtained in 2017 by Kriventsov and Lin in \cite{KriventsovLin2019:degenerate} by an argument that also applies to the  degenerate vectorial problem. \medskip 

\noindent Another important tool for the understanding of the regularity of the flat free boundaries in the degenerate case was introduced in \cite{MazzoleniTerraciniVelichkov:GAFA}, where it was shown that the free boundary condition for the gradient of the vector $U$ on $\partial\Omega_U$ can be written in terms of the scalar function $|U|$. Precisely,
$$|\nabla |U||=\sqrt\Lambda\quad\text{on}\quad\partial\Omega_U\cap D.$$
By exploiting this viscosity formulation of the free boundary condition, two different proofs of the smoothness of $Reg$ were obtained in \cite{MazzoleniTerraciniVelichkov:AnalPDE} and \cite{DeSilvaTortone2020:ViscousVectorialBernoulli} for the degenerate vectorial problem (in any dimension $d\ge 2$). The proof in \cite{MazzoleniTerraciniVelichkov:AnalPDE} exploits the properties of NTA domains and the De Silva's regularity result \cite{DeSilva:FreeBdRegularityOnePhase} for the one-phase problem, while in \cite{DeSilvaTortone2020:ViscousVectorialBernoulli} the regularity of the flat free boundaries was obtained via an improvement of flatness argument based on a Partial Harnack inequality for viscosity solutions; this approach was later applied to the thin vectorial problem in  \cite{DeSilvaTortone:thin}.
\medskip

\noindent{\it The singular set in dimension $d\ge 3$.} Finally, we notice that for general degenerate (vectorial and shape optimization) problems the structure of the singular set in dimension $d>2$ is currently an open problem.

\subsection{Non-variational Bernoulli free boundary systems}\label{s:subs-system}
The vectorial (nondegenerate and degenerate) Bernoulli problem \eqref{e:vectorial-bernoulli} discussed in the previous \cref{sub:intro:vectorial} are both of variational nature, that is, there is a functional (in our case \eqref{e:vectorial-bernoulli-functional}) defined on the class of vector valued functions $U:\R^d\to\R^k$ whose minimizers are solutions to \eqref{e:vectorial-bernoulli}. The same is also true for other free boundary systems. For instance, let $D$ be a bounded open set in $\R^d$, $k\ge 1$ and let $A_{ij}$ be a family of symmetric $d\times d$ matrices with $A_{ij}=A_{ji}$. Then, solutions $U=(u_1,\dots,u_k):D\to\R^k$ of the free boundary system
\begin{equation*}
\begin{cases}
\text{div}\Big(\sum_{i=1}^kA_{ij}\nabla u_i\Big)=0&\quad\text{in}\quad\Omega:=\{|U|>0\}\quad\text{for every}\quad j=1,\dots,k,\\
\sum_{i=1}^k\sum_{i=1}^k\nabla u_j\cdot A_{ij}\nabla u_i=1&\quad\text{on}\quad\partial\Omega\cap D,
\end{cases}
\end{equation*}
can be obtained by minimizing the functional 
\begin{equation*}
\int_{D}\Bigg(\sum_{i=1}^k\sum_{j=1}^k\nabla u_i\cdot A_{ij}\nabla u_j\Bigg)\,dx+|\{|U|>0\}\cap D\}|.
\end{equation*}
Notice that if the free boundary problem admits a variational formulations, then the boundary condition on $\partial\Omega\cap D$ and the system solved in $\Omega$ are obtained from the same expression, so they cannot be assigned independently. In particular, the free boundary condition on $\partial\Omega$ is elliptic if and only if the functional is elliptic. This suggests that solutions to free boundary problems with non-elliptic free boundary condition cannot be obtained my minimizing a functional as the functional itself will not be elliptic. 

A simple example of a non-variational free boundary system is the following one, defined for vector valued functions $U=(u,v):D\to\R^2$ in a fixed bounded open set $D\subset\R^d$: 
\begin{equation}\label{e:free-boundary-system-intro-0}
\begin{cases}
\begin{array}{rl}
\Delta u=0&\quad\text{in}\quad\Omega:=\{|U|>0\},\\
\Delta v=0&\quad\text{in}\quad\Omega:=\{|U|>0\},\\
\nabla u\cdot \nabla v=1&\quad\text{on}\quad\partial\Omega\cap D.
\end{array}
\end{cases}
\end{equation}
Systems of this form arise in optimal control and inverse problems in which the aim is to detect an unknown object (the complementary of the domain $D\setminus\Omega$) by analyzing the trace and the gradient on the fixed boundary $\partial D$ of the solution of an overdetermined PDE problem in $D\setminus\Omega$ (see for instance \cite{AguileraCaffarelliSpruck,BruggerHarbrechtTausch:NumericalSolution,ButtazzoMaialeVelichkov:Lincei,ButtazzoVelichkov}). A corresponding variational shape optimization problem can be formulated as follows: given a domain $D\subset\R^d$, a force term $f:D\to\R$, and a cost function $j\colon \R \times D \to \R$, minimize the  functional
\be\label{e:J}
J(A):=\int_A j(u_A,x)\,dx\,,
\ee
among all open sets $A\subset D$, where the state variable $u_A$ is the unique solution of the PDE
$$
-\Delta u_A = f \quad \mbox{in }A\ ,\qquad u=0\quad \mbox{on }\partial A\,.
$$
Even if the above problem is formulated only in terms of the scalar function $u_A$, it is naturally associated to a free boundary system for the couple $(u_A,v_A)$,
where $v_{A}\in H^1_0(A)$ is the so-called {\it adjoint variable} obtained by solving
$$
\Delta v_{A} = \frac{\partial j}{\partial u}(u_{A}(x),x)\quad\mbox{in } {A}\ ,\qquad v_{A}=0\quad\mbox{on } \partial {A}.
$$

\begin{minipage}{1.5\linewidth}
   \hspace{-0.023\linewidth}
\begin{minipage}{0.38\linewidth}
\vspace{0.1cm}
Indeed, if $\Omega$ minimizes \eqref{e:J} among all sets of given measure, by formally computing the domain variations of the functional $J$, it is possible to derive the validity of the free boundary condition
\be\label{e:free-boundary-fg}
|\nabla u_{\O}||\nabla v_{\O}| = \Lambda + j(0,x) \quad\mbox{on }\partial{\O}\cap D,
\ee
where $\Lambda$ is a positive constant (see for instance \cite{BruggerHarbrechtTausch:NumericalSolution}).\medskip

A free boundary condition of the type \eqref{e:free-boundary-fg} was first studied by Aguilera, Caffarelli and Spruck in \cite{AguileraCaffarelliSpruck} for an optimization problem with heat conduction.\vspace{0.1cm}\\
  \end{minipage}
\hspace{0.01\linewidth}
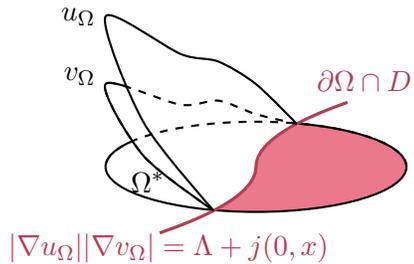
 \begin{minipage}{0.3\linewidth}
 \vspace{-0.1cm}
 \begin{tikzpicture}[scale=2.0]
 \def\a{-0.53}
 \def\b{0.52}
 \def\p{0.55}
 \def\q{1.1}
  \begin{scope}
  \clip (0,0) ellipse (1cm and 0.3cm);
\filldraw[red2!55, domain=-0.8:0.78]  (0,-1) -- plot (\x^2,0.55*\x ) -- (1.5,0) -- (0,-1) -- cycle;
 \end{scope}
    \draw[black] (-0.72,-0.1) node {$\Omega^*$};
       \draw[black,thick] (1,0) arc (0:180:1cm and 0.3cm);
  \draw[black,thick] (-1,0) arc (180:360:1cm and 0.3cm);
   \fill [white] plot [smooth] coordinates {(\a^2,0.55*\a) (-0.75,0.3*0.25)  (-1,\p) (-0.5,0.75*\p) (-0.25,0.8*\p) (0,0.35) (\b^2,0.55*\b) (0.27^2,0.55*0.3) (0,0) (-0.29^2,-0.55*0.3) (-0.5^2,-0.55*0.5) (\a^2,0.55*\a)};
\draw [black,thick] plot [smooth] coordinates {(\a^2,0.55*\a) (-0.75,0.3*0.25)  (-1,\p) (-0.5,0.75*\p) (-0.25,0.8*\p) (0,0.35) (\b^2,0.55*\b)};
   \fill [white] plot [smooth] coordinates {(\a^2,0.55*\a) (-0.75,0.3)  (-1,1) (-0.5,0.75) (-0.25,0.7) (0,0.55) (\b^2,0.55*\b) (0.5^2,0.55*0.5)
    (0.27^2,0.55*0.3) (0,0) (-0.29^2,-0.55*0.3) (-0.5^2,-0.55*0.5) (\a^2,0.55*\a) };
   \draw [black,thick] plot [smooth] coordinates {(\a^2,0.55*\a) (-0.75,0.3)  (-1,1) (-0.5,0.75) (-0.25,0.7) (0,0.55) (\b^2,0.55*\b) };
      \draw[color=red2!55!gray, very thick, domain=-0.8:0.78] plot (\x^2,0.55*\x ) node[anchor  =south ]{$\,\quad\partial \O\cap D$};
          \draw[black,thick,dashed] (1,0) arc (0:180:1cm and 0.3cm);
   \draw [black,thick,dashed] plot [smooth] coordinates {(\a^2,0.55*\a) (-0.75,0.3*0.25)  (-1,\p) (-0.5,0.75*\p) (-0.25,0.8*\p) (0,0.35) (\b^2,0.55*\b)};

   \draw (-1,1) node[anchor=east] {$u_{\O}$};
\draw (-1,0.6) node[anchor=east] {$v_{\O}$};
   \draw[color=red2!65!gray] (-0.58,-0.54) node {$|\nabla u_{\O}||\nabla v_{\O}|= \Lambda+j(0,x)$};
   \end{tikzpicture}
\end{minipage}
\end{minipage}
In the framework of \cite{AguileraCaffarelliSpruck}, $u_A$ is a temperature distribution in the domain $D$, which is harmonic along the insulating material $A\subset D$ and vanishes in $D\setminus A$; the cost functional $J$ is the heat flow through $\partial D$. Thus, if for $A\subset D$ we consider the state variables 
$$
\begin{cases}
\Delta u_A = 0 &\mbox{in }A,\\
u_A=\varphi &\mbox{on }\partial D,\\
u_A=0 &\mbox{on }\partial A \cap D,
\end{cases}\qquad
\begin{cases}
\Delta v_A = 0 &\mbox{in }A,\\
v_A =\phi &\mbox{on }\partial D,\\
v_A=0 &\mbox{on }\partial A \cap D,
\end{cases}
$$
we can rewrite the cost functional as
$$
J(A):=\int_{\partial D} \phi\,\nabla u_A\cdot \nu_{\partial D}\,d\mathcal{H}^{d-1} = \int_D \nabla u_A\cdot \nabla v_A\,dx.
$$
In \cite{AguileraCaffarelliSpruck} it was proved that the boundary of the optimal set is smooth up to a closed set of zero $(d-1)$-Hausdorff measure and that on the regular part of the free boundary a condition of the form \eqref{e:free-boundary-fg} is satisfied. \medskip

Recently, in the series of papers in collaboration \cite{MaialeTortoneVelichkov:RMI,MaialeTortoneVelichkov:AnnSNS,ButtazzoMaialeMazzoleniTortoneVelichkov:ARMA} we developed a program for the study of the regularity of the free boundary of optimal sets for integral type functionals of the form \eqref{e:J}. First, in \cite{MaialeTortoneVelichkov:RMI}, we proved an epsilon-regularity theorem for continuous functions satisfying  \eqref{e:free-boundary-system-intro-0} in viscosity sense. Then, in \cite{MaialeTortoneVelichkov:AnnSNS}, we proved a general Boundary Harnack Principle for domains satisfying a list of conditions, which are typically fulfilled by optimal sets and solutions to vectorial free boundary problems.  Finally,  we completed the program in \cite{ButtazzoMaialeMazzoleniTortoneVelichkov:ARMA} for a class of model shape optimization problems with cost functional of the form 
$$j(u,x):=-g(x)u.$$
Precisely, in \cite{ButtazzoMaialeMazzoleniTortoneVelichkov:ARMA}, we proved that the optimal sets satisfy the conditions from \cite{MaialeTortoneVelichkov:AnnSNS} and we used this information to show that the state functions $u_\Omega$ and $v_\Omega$, and the optimal set $\Omega$, are viscosity solutions to \eqref{e:free-boundary-system-intro-0}. Again in \cite{ButtazzoMaialeMazzoleniTortoneVelichkov:ARMA} we introduced the notion of stale solutions to the one-phase Bernoulli problem, which we used to show that the singular set on $\partial\Omega$ is of Hausdorff dimension $d-5$. Finally, in the paper \cite{MazzoleniTortoneVelichkov:JConvexAnal} in collaboration with Mazzoleni, we applied the strategy developed in \cite{ButtazzoMaialeMazzoleniTortoneVelichkov:ARMA} to improve the Hausdorff dimension estimate from  \cite{AguileraCaffarelliSpruck}.

\subsection{Structure of the survey} In \cref{s:eigenvalues}, we describe the main results achieved in the context of vector-valued problems involving Dirichlet eigenvalues (see \cite{KriventsovLin2018:nondegenerate,MazzoleniTerraciniVelichkov:GAFA,KriventsovLin2019:degenerate,Tortone:shapefractional}), with particular attention to the differences compared to their scalar counterparts.
Then, in \cref{s:vectorial-bernoulli} we considered the minimization problem associated to \eqref{e:vectorial-bernoulli-functional} and we illustrate the various methodologies introduced for the study of $\eps$-regularity in \cite{CaffarelliShahgholianYeressian2018:vectorial,MazzoleniTerraciniVelichkov:AnalPDE,DeSilvaTortone2020:ViscousVectorialBernoulli}.
In \cref{s:vectorial-fg} we discuss the strategy developed in \cite{MaialeTortoneVelichkov:RMI,MaialeTortoneVelichkov:AnnSNS,ButtazzoMaialeMazzoleniTortoneVelichkov:ARMA}; we describe the main feature of the epsilon-regularity theorem for \eqref{e:free-boundary-system-intro-0} and we 
explain the principles behind the study of singular points for stable one-phase solutions. We also include some open problems.

\section{Optimal domains for functionals involving Dirichlet eigenvalues}\label{s:eigenvalues}
A shape optimization problem is a variational problem of the form
\be\label{e:general.shape}
\min\big\{J(A)\ :\ A\in\mathcal{A}\big\},
\ee
where $\mathcal{A}$ is an admissible class of sets, and where $ J:\mathcal{A}\to\R$ is a given functional on $\mathcal A$ usually depending on the solution of a PDE on the domains $A\in \mathcal A$. Such minimization problems arise in different models in Biology, Engineering and Physics (see for instance \cite{BucurButtazzoBook,HenrotBook,HP18,Velichkov:BookShapeOptimization} for an overview) and have been extensively studied from both numerical and theoretical points of view.\\

A special class of shape functionals is the one of so-called {\it spectral functionals}, that is, functionals depending on the eigenvalues of the Dirichlet Laplacian as
$$J(A)=F\big(\lambda_1(A),\dots,\lambda_k(A)\big),$$
where $F:\R^{k}\to\R$ is a real-valued function and $\lambda_i(A)$ denotes the $i$-th eigenvalue of the Dirichlet Laplacian on the set $A$ counted with the due multiplicity. For instance, given a set $D\subseteq \R^d$ and constants $\Lambda,c>0$, two classical shape optimization problems are the following:
\begin{enumerate}
    \item[(i)] the penalized problem
    \be\label{e:penalized}\min\Big\{F\big(\lambda_1(A),\dots,\lambda_k(A)\big)+\Lambda|A|\ :\ A\subset D\mbox{ open}\Big\},\ee
    \item[(ii)] the measure-constrained problem
    \be\label{e:constrained}\min\Big\{F\big(\lambda_1(A),\dots,\lambda_k(A)\big)\ :\ A\subset D \mbox{ open},\,|A|=c\Big\},\ee
\end{enumerate}
where in both cases the cost function is defined through a reasonable\footnote{For general functions $F$ the existence of a solution in the class of quasi-open sets was first proved by Buttazzo and Dal Maso in \cite{buttazzodalmaso} for $F$ non-decreasing in each variable and lower semi-continuous, under the assumption that the candidate sets are all contained in a bounded open set $D\subset \R^d$. This last assumption was later removed by Bucur in \cite{bucur-kth} and Mazzoleni and Pratelli in \cite{mazzolenipratelli}.} function $F\colon \R^k \to \R$. 
The first results on the characterization of the optimal shapes, for the first and the second Dirichlet eigenvalues, go back to the works of Faber-Krahn (1922) and Krahn-Szeg\"o (1923) where it was shown that the only minimizers in $\R^d$ are balls and unions of disjoint balls whose radius depends on $c$ and $\Lambda$. It is today known, mainly thanks to numerical experiments, that in general the minimizers have more complicated geometry (see for instance \cite{OudetNumerical}). 
One way to approach this problem, and to study the interplay between the geometry of the domains and the spectrum of their Dirichlet Laplacian, is by analyzing the local structure of the optimal sets for the above classes of problems. Today the regularity of the optimal sets is well-understood in the case of the functionals $F(\lambda_1,\dots,\lambda_k)=\lambda_1$ and $F(\lambda_1,\dots,\lambda_k)=\lambda_2$ (in the case $D\subset\R^d$ since otherwise the problem is trivial). The regularity of the optimal sets for the first eigenvalue $\lambda_1$ in a box $D\subset\R^d$ was studied in \cite{BrianconLamboley:RegularityShapeFirstEigenLaplacian} and \cite{RussTreyVelichkov:CalcVarPDE}, while the regularity of the optimal sets the second eigenvalue $\lambda_2$ in a box $D\subset\R^d$ was obtained recently in \cite{MazzoleniTreyVelichkov:AnnIHP}.\medskip

Recently, several authors have investigated the regularity of optimal shapes for cost functionals involving higher-order eigenvalues, where the nontrivial interactions between distinct eigenfunctions give rise to complex and intriguing phenomena. In order to make the discussion self-contained, we recall here the definition of \emph{non-degenerate} and \emph{degenerate} functional from \cite{KriventsovLin2018:nondegenerate,KriventsovLin2019:degenerate}.

\begin{definition}\label{d:non-deg-deg}
Let $F\colon \R^k \to [0,+\infty)$ be a Lipschitz function non-decreasing in each variable. Then: 
\begin{enumerate}
    \item $F$ is said to be \emph{non-degenerate} if either $F$ is locally by-Lipschitz or $F\in C^1$ and there exists $\mu>0$ such that $\partial_{\lambda_i}F\geq \mu >0$ for every $i=1,\dots,k$, a model problem being the minimization of the sum of the first $k$ Dirichlet eigenvalues (the case $k=1$ studied in \cite{BrianconLamboley:RegularityShapeFirstEigenLaplacian,RussTreyVelichkov:CalcVarPDE} clearly falls in this class);\medskip
    \item $F$ is said to be \emph{degenerate} if $F$ is locally $C^1$ or it is strictly increasing in each variable. 
  Here, the model problem is the one associated to the minimization of $k$-th Dirichlet eigenvalue, with $k>1$ (also the case $k=2$ studied in \cite{MazzoleniTreyVelichkov:AnnIHP} falls in this class).
\end{enumerate}
\end{definition}
For a detailed analysis of the role played by each assumption in the degenerate setting, we refer the reader to \cite[Section 2]{KriventsovLin2019:degenerate}.

\subsection{Qualitative properties of eigenfunctions and optimal domains} Before presenting the main result, it is important to introduce the notions of shape subsolution (inward optimality) and shape supersolution (outward optimality). As already observed by Bucur in \cite{bucur-kth}, several qualitative properties of the optimal shapes and of the associated state functions can be derived by simply employing these two concepts. Given a class of measurable sets $\mathcal A$, a functional $ J:\mathcal A\to\R$, and a set $\Omega\in\mathcal A$, we say that:
\begin{enumerate}
    \item[(i)] $\Omega$ is {\it a shape subsolution} (or {\it inwards minimizing}) if
    $$
    J(\O)\leq J(A) \quad \mbox{for every }A\in \mathcal{A},\,A\subseteq \O;
    $$
    \item[(ii)] $\Omega$ is {\it a shape supersolution} (or {\it outwards minimizing}) if 
    $$
    J(\O)\leq J(A) \quad \mbox{for every }A\in \mathcal{A},\,\O\subseteq A.
    $$
\end{enumerate}
In general, the inwards minimizing property implies that the sets are bounded and have finite perimeter, while the outwards minimizing property allows to prove that the state functions (for instance, the eigenfunctions) are Lipschitz continuous.
\begin{theorem}\label{t:subsolution}
Let $F\colon \R^k \to [0,+\infty)$ be a Lipschitz function non-decreasing in each variable. Then, every shape subsolution $\O$ of \eqref{e:penalized} and \eqref{e:constrained} is a bounded set with finite perimeter. Moreover, there exist two dimensional constants $\eps_d,r_d>0$ such that 
\be\label{e:interior-density}
|\O\cap B_r(x_0)|\geq \eps_d|B_r|\quad\mbox{for every }x_0 \in \partial\O, r\leq r_d.
\ee
\end{theorem}
\begin{proof}
Let us split the discussion in two cases:
\begin{enumerate}
    \item[(i)] {\it $F$ is non-degenerate.} 
By the local bi-Lipschitz continuity of $F$, one can deduce that every subsolution $\O$ is in fact a shape subsolution for the functional $A\mapsto \lambda_k(A)+\Lambda'|A|$, for some constant $\Lambda'>0$, among quasi-open subsets of $D$. The claimed result follows by showing that $\O$ is also a shape subsolution for the torsion functional as established in \cite[Theorem 3.1]{bucur-kth}, which in turn allows us to apply \cite[Theorem 2.2]{bucur-kth}. The interior density estimate follows by applying \cite[Proposition 3.11]{BucurVelichkov-multiphase}.\medskip
\item[(ii)] {\it $F$ is degenerate.} By \cite[Lemma 4.2]{KriventsovLin2019:degenerate} we can reduce the problem to the case of shape subsolution for the torsion functional. Then, the proof follows as before (see also \cite[Corollary 5.2 and Theorem 6.1.]{KriventsovLin2019:degenerate} for a direct proof the density estimate).\qedhere
\end{enumerate}
\end{proof}
\begin{theorem}\label{t:supersolution}
Let $F\colon \R^k \to [0,+\infty)$ be a Lipschitz function non-decreasing in each variable. Then, every shape supersolution $\O$ of \eqref{e:penalized} and \eqref{e:constrained} is an open set. Moreover
\begin{enumerate}
    \item[(i)] if $F$ is non-degenerate, the first $k$ normalized eigenfunctions $u_i\in H^1_0(\O)$, extended by zero over $\R^d\setminus \O$, are Lipschitz continuous on $\R^d$;
\item[(ii)] If $F$ is degenerate, then there exists an orthonormal family of eigenfunctions $u_j\in H^1_0(\Omega^\ast)$, for $j = 1, \dots, m$, each corresponding to eigenvalues not exceeding $\lambda_k(\Omega^*)$. Moreover, all these eigenfunctions are Lipschitz continuous on $\R^d$.
\end{enumerate}
\end{theorem}
\begin{proof}
In \cite{BucurMazzoleniPratelliVelichkov:ARMA} (see also \cite{DavidToro2015:AlmostMinimizers,BrianconHayouniPierre}) it was shown that for every eigenvalue $\lambda_i$ at which $\partial_{\lambda_i}F>0$ there is an associated eigenfunction $u_i$ which admits a Lipschitz continuous extension to the whole $\R^d$. This immediately implies claim (i). In the degenerate case (ii), the result from \cite{BucurMazzoleniPratelliVelichkov:ARMA} still allows to select a family of Lipschitz continuous eigenfunctions $u_{i_1},\dots, u_{i_\ell}$, but this doesn't exclude the presence of another eigenfunction $u_j$ on $\Omega$, which is not continuous. In fact, it might happen that the set $\Omega$ is not even an open set due to a presence of subset $K\subset\{u_{i_1}=\dots=u_{i_\ell}=0\}\cap\Omega$, which has positive capacity, but remains invisible to the selected eigenfunctions $u_{i_1},\dots, u_{i_\ell}$. In this case, the problem was solved in \cite{KriventsovLin2019:degenerate} by selecting a special minimizer $\Omega$ and by combining \cite[Lemma 5.1]{KriventsovLin2019:degenerate} with the limiting procedure in  \cite[Section 6]{KriventsovLin2019:degenerate}.
\end{proof}
To make the discussion more readable, from now on we will focus on the two most representative cases of non-degenerate and degenerate functionals, respectively \eqref{e:somma-intro} and \eqref{e:k-intro}. 
We will highlight relevant details from the general case whenever significant differences arise compared to these model cases.

\subsection{Regularity of optimal shapes for $\sum_{i=1}^k\lambda_i(A)$} 
Let $c\in \R$ and $\O\subset \R^d$ be an optimal set for the measure constrained problem
    \be\label{e:sum-D}
\min\left\{ \sum_{i=1}^k \lambda_i(A)\ :\ A\subset D\,,\ |A|=c\right\},
    \ee
where the sets $A$ can be chosen in the class of quasi-open sets or in the class of Lebesgue measurable sets (see \cite{Velichkov:BookShapeOptimization}).     
Then, the following regularity result holds.

\begin{theorem}\label{t:main-reg-sum}
 The minimizers of \eqref{e:sum-D} are open sets and the boundary of any optimal set $\O$ in $D$ can be decomposed as the disjoint union of a regular and a singular part
   $$
   \partial \O \cap D = \mathrm{Reg}(\partial \O) \cup \mathrm{Sing}(\partial \O),
   $$
   where 
   \begin{itemize}
        \item[(i)] the regular part $\mathrm{Reg}(\partial \O)$ is a smooth manifold of dimension $(d-1)$. Moreover, the vector $U = (u_1,\dots,u_k)$ of the normalized eigenfunctions satisfies
        $$
        |\nabla |U||=\sqrt{\Lambda} \quad \mbox{on }\mathrm{Reg}(\partial \O),\qquad \mbox{with }\,\Lambda:=\frac{2}{d}\sum_{i=1}^k\lambda_i(\O);
        $$
        \item[(ii)] the singular part $\mathrm{Sing}(\partial \O)$ is a closed subset of $\partial \O \cap D$. Moreover, there exists a dimensional threshold $N \in \{5,6,7\}$ such that the following properties holds
\begin{itemize}
\item[(1)] If $d<N$, then $\text{\rm Sing}(\partial \Omega)$ is empty.
\item[(2)] If $d\ge N$, then the Hausdorff dimension of $\text{\rm Sing}(\partial \Omega)$ is at most $d-N$, namely
\[
\HH^{d-N+\eps}\big(\text{\rm Sing}(\partial \Omega)\big)=0\quad\text{for every}\quad \eps>0.
\]
\end{itemize}
The dimensional threshold $N$ coincides with the dimensional threshold $d^*$ for the Alt-Caffarelli's functional \eqref{e:ac-functional}.
        \end{itemize}
\end{theorem}
This result was first proved in  \cite[Theorem 1.3]{MazzoleniTerraciniVelichkov:GAFA} and in \cite{KriventsovLin2018:nondegenerate} in the case of general non-degenerate functionals; we also refer to \cite{CaffarelliShahgholianYeressian2018:vectorial}, which appeared simultaneously with \cite{MazzoleniTerraciniVelichkov:GAFA} and \cite{KriventsovLin2018:nondegenerate}, and to the works of Trey \cite{Trey:COCV,Trey:IHP} for the more general case of operators in divergence form. In \cite{KriventsovLin2018:nondegenerate} the authors generalized \cref{t:main-reg-sum} to the case of general non-degenerate functionals $F$. To aid the reader in comparing these different results and possible future extensions,
we suggest the following Kriventsov-Lin's scheme about the possible assumptions on $F$: 
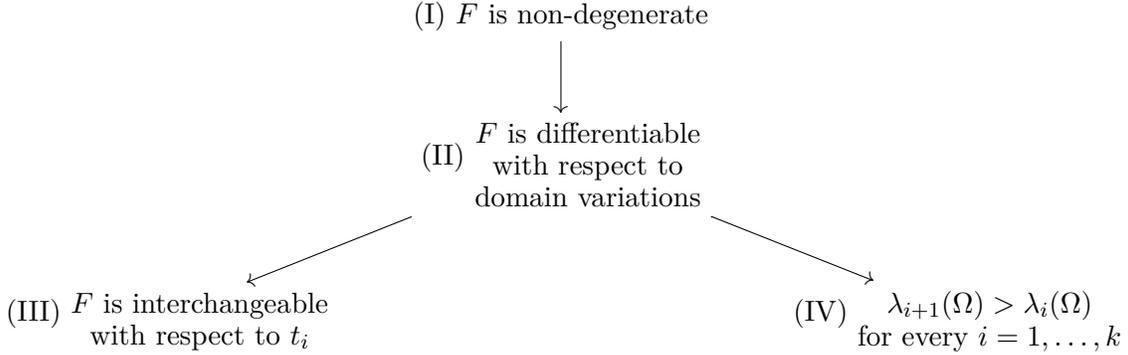
\[
\xymatrix@M=4pt{ 
& \text{(I) $F$ is non-degenerate}\ar[d]\\
& \text{(II) } \txt{$F$ is differentiable \\ with respect to \\ domain variations} \ar[dl] \ar[dr]\\
	\text{(III) }  \txt{$F$ is interchangeable\\ with respect to $t_i$} & & \text{(IV) } \txt{$\lambda_{i+1}(\O)>\lambda_{i}(\O)$ \\ for every $i=1,\dots,k$}}
\]
If $F$ is non-degenerate in the sense of \cref{d:non-deg-deg}, then \cref{t:main-reg-sum} holds true for minimizers of $F$ by taking as dimensional threshold $N=3$ (see \cite[Theorem 1.1]{KriventsovLin2018:nondegenerate}). Again in \cite{KriventsovLin2018:nondegenerate}, the optimal estimate $N=d^\ast$ on the singular set was recovered by assuming the condition (IV). 
It is natural to ask whether the optimal dimension bound holds when assuming only (II). See \cref{e:open-singular} for further discussion.

\subsubsection{Free boundary reformulation} In the regularity theory, when dealing with a shape optimization problem, a common strategy is to reformulate it as a free boundary problem (see for instance \cite{BucurVelichkovSurvey}). This reformulation allows to transform the questions about the regularity of the boundary of an optimal shape into questions about the regularity of the support of a function minimizing a suitable energy. In the case of \eqref{e:sum-D} it is natural reformulate the problem in terms of the vector of normalized eigenfunctions $u_1,\dots,u_k$. For simplicity, we take $D=\R^d$, since in this case \eqref{e:sum-D} is equivalent (see \cite[Lemma 2.4]{MazzoleniTerraciniVelichkov:GAFA}) to the following problem: 
\begin{equation}\label{e:sum-R-d-penalized}
\min\left\{ \sum_{i=1}^k \lambda_i(A) + \Lambda |A|\ :\ A\subset \R^d \mbox{ open}\right\},\quad\mbox{with }\,\Lambda:=\frac{2}{dc}\sum_{i=1}^k\lambda_i(\O).
\end{equation}
Since the Dirichlet Laplacian on $A$, $|A|<+\infty$, is a positive and self-adjoint operator with compact resolvent, we have the following variational characterization for the sum of its  eigenvalues: 
\begin{equation}\label{e:vectorial-reformulation-sum-lambda-k-2}
\sum_{i=1}^k \lambda_i(A)=\mathrm{min}\left\{\int_{\R^d}|\nabla V|^2\,dx\  : \ V \in H^1_0(A;\R^k),\, \int_{\R^d}v_iv_j\,dx = \delta_{ij}\right\}.
\end{equation}
Thus, the vector $U:=(u_1,\dots,u_k)$ of the first $k$ normalized eigenfunctions of an optimal set $\O$ for \eqref{e:sum-R-d-penalized} is a solution of the problem
\begin{equation}\label{e:vectorial-reformulation-sum-lambda-k}
\mathrm{min}\left\{\int_{\R^d}|\nabla V|^2\,dx + \Lambda |\{|V|>0\}| \colon V \in H^1(\R^d;\R^k),\, \int_{\R^d}v_iv_j\,dx = \delta_{ij}\right\}.
\end{equation}
Vice-versa, if $U\in H^1(\R^d;\R^k)$ minimizes \eqref{e:vectorial-reformulation-sum-lambda-k}, then the set $\O_U:=\{|U|>0\}$ is optimal for \eqref{e:sum-R-d-penalized}.

\subsubsection{Almost-minimality of the vectors of eigenfunctions} By applying a Gram-Schmidt
procedure (see \cite[Lemma 2.5]{MazzoleniTerraciniVelichkov:GAFA}), we can test the optimality of $U$ with competitors $\widetilde U$ that do not satisfy the $L^2$-orthogonality constraint in \eqref{e:vectorial-reformulation-sum-lambda-k}. Precisely, $U$ satisfies the following almost-minimality condition (see \cite[Proposition 2.1]{MazzoleniTerraciniVelichkov:GAFA}): for every 
$\delta \geq \lVert U\rVert_{L^\infty}$ there exist $K,\eps>0$ such that 
$$
\int_{\R^d}|\nabla U|^2\,dx + \Lambda |\{|U|>0\}| 
\leq \big(1+ K \lVert U-\widetilde{U}\rVert_{L^1}\big)\int_{\R^d}|\nabla \widetilde{U}|^2\,dx + \Lambda |\{|\widetilde{U}|>0\}|,
$$
for every $\widetilde{U}\in H^1(\R^d;\R^k)$ such that $\lVert \widetilde{U}\rVert_{L^\infty} \leq \delta $ and $\lVert U-\widetilde{U}\rVert_{L^1}\leq \eps$. Thus, the study of regularity of the minimizers of \eqref{e:vectorial-reformulation-sum-lambda-k} boils down to the study of the almost-minimizers of the vectorial Bernoulli problem (with non-negative first component). We notice that almost-minimizers of Bernoulli free boundary problems have been studied for instance in \cite{DavidToro2015:AlmostMinimizers,DavidEngelsteinToro2019:AlmostMinimizers}.


\subsubsection{Blow-up sequences and blow-up limits} Starting from this almost-minimality condition, one can deduce important properties such as the Lipschitz continuity of $U$, the non-degeneracy of the norm $|U|$ near the free boundary, the validity of a Weiss-type monotonicity formula (see \cite{MazzoleniTerraciniVelichkov:GAFA}), as well as the non-degeneracy of the first eigenfunction $u_1$ (for which we refer to \cite{KriventsovLin2018:nondegenerate}) and the fact that $\O_U=\O_{u_1}$. These results lead to a complete characterization of the blow-up limits of $U$ with fixed centers on $\partial \O_U$ (see \cite{MazzoleniTerraciniVelichkov:GAFA}). Precisely, given $x_0 \in \partial \O_U$ and $r>0$ the blow-up sequence
\be\label{e:blow-up-sequence}
U_{x_0,r}(x):=\frac{1}{r}U(x_0+rx)
\ee
is uniformly Lipschitz and locally uniformly bounded in $\R^d$ and so, up to a subsequence, it converges locally uniformly in $\R^d$ to a $1$-homogeneous non-trivial local minimizer $U_0 \in H^1_\loc(\R^d;\R^k)$ of \eqref{e:vectorial-bernoulli-functional}. Moreover, we can show that 
\be\label{e:blow-up-xi}
U_0(x) = |U_0(x)|\xi \quad \mbox{in }\R^d,\mbox{ with }\xi\in \partial B_1\subset \R^k.
\ee
Ultimately, such characterization implies that the norm $|U_0|$ is a $1$-homogeneous local minimizer of the Alt-Caffarelli functional \eqref{e:ac-functional} (see \cite[Section 4]{MazzoleniTerraciniVelichkov:GAFA}).\\

In the case of general non-degenerate functionals, this program has been generalized in \cite[Section 3]{KriventsovLin2018:nondegenerate} (see also \cite[Remark 1.2]{KriventsovLin2018:nondegenerate}). In this setting, it is not possible to prove directly that the blow-up limits are local minimizers of a vectorial functional of the form \eqref{e:vectorial-bernoulli-functional}.\\
Indeed, in \cite[Lemma 3.2 and Proposition 8.1]{KriventsovLin2018:nondegenerate}, the authors characterized the blow-up limits for functionals satisfying one of the following two properties:
\begin{itemize}
    \item[(i)] {\it Property S: all the eigenvalues of the optimal domain are simple.}\\ If Property S holds, then the blow-up limits $U_0$ at $x_0\in\partial \O_U$ are $1$-homogeneous local minimizers of 
    $$
    V \mapsto \sum_{i=1}^k \partial_{\lambda_i}F(\O)\int_{\R^d} |\nabla v_i|^2 \,dx + \Lambda |\{|V|>0\}|. 
    $$
    \item[(ii)] {\it Property E: there exists $\xi\in \R^k$ with non-negative components and a collection of orthonormal eigenfunctions $U:=(u_1,\dots,u_k)$ such that}
    $$
    \sum_{i=1}^k \xi_i|\nabla u_i|^2 = \Lambda \quad\mathcal{H}^{n-1}\mbox{-a.e. on }\partial \O_U.
    $$
    If Property E holds, then every blow-up limit is a $1$-homogeneous vector $U_0 :=(u_{0,1},\dots,u_{0,k})$ satisfying   
    $$
    \Delta U_0 = - (\nabla U_0\cdot \nu) \,d\mathcal{H}^{n-1}\mres\partial\O_U,\qquad\sum_{i=1}^k \xi_i|\nabla u_{0,i}|^2 = \Lambda \quad\mathcal{H}^{n-1}\mbox{-a.e. on }\partial\O_U.
    $$
\end{itemize}
Notice that, in both the previous cases, for every blow-up limit $U_0=(u_{0,1},\dots,u_{0,k})$ we have that $\O_{U_0}=\O_{u_{0,1}}$ and that all the components of $U_0$ are proportional, i.e., \eqref{e:blow-up-xi} holds true.\smallskip

In \cite{KriventsovLin2018:nondegenerate} it was proved that every minimizer satisfies property $E$, which is a surprising result, given that the functional $F(\lambda_1,\dots,\lambda_k)$ is not necessarily differentiable with respect to domain variations. However, it is not clear if all minimizers satisfy property S (see \cite[Lemma 3.4]{KriventsovLin2018:nondegenerate} where Property E is deduced from  Property S). By looking to the previous characterization, it is  immediate to check that
\be\label{e:property-SE}
\begin{aligned}
\text{S}&\,\,\longmapsto\,\,\text{Blow-ups are homogeneous global minimizers of the Alt-Caffarelli's functional}\\
\text{E}&\,\,\longmapsto\,\,\text{Blow-ups are homogeneous stationary solutions to the one-phase problem.}
\end{aligned}
\ee

\begin{remark}
A  crucial property of  the non-degenerate functionals is that the first eigenfunctions on optimal domains are non-degenerate (see \cite[Lemma 2.5]{KriventsovLin2018:nondegenerate}). This in turn implies the validity of the following exterior density estimate (see \eqref{e:interior-density}), i.e., there exist $\eps_d,r_d>0$ such that 
\be\label{e:exterior-density}
|\O_U\cap B_r(x_0)|\leq (1-\eps_d)|B_r|\quad\mbox{for all}\quad x_0 \in \partial\O_U\quad\text{and all}\quad r\leq r_d,
\ee
which excludes the presence of cusp singularities. Another immediate consequence is that the optimal sets in $\R^d$ are connected
(\cite[Corollary 4.3]{MazzoleniTerraciniVelichkov:GAFA} and \cite[Lemma 2.6]{KriventsovLin2018:nondegenerate}).
\end{remark}

\subsubsection{Decomposition of the free boundary} We can decompose the free boundary in terms of the $1$-homogeneous solution that appears at blow-up, precisely:
\begin{enumerate}
    \item[(i)] the regular part $\mathrm{Reg}(\partial \O)$ as the collection of points at which there is (at least one) blow-up limit $U_0$ of the form
    \be\label{e:flat-solution}
    U_0(x)=\sqrt{\Lambda}(x\cdot \nu)^+ \xi 
    \ee
    where $\Lambda>0$,  $\nu \in \R^d, |\nu|=1$ and $\xi \in \R^k$, $|\xi|=1$; 
    \item[(ii)] the singular part $\mathrm{Sing}(\partial \O)$
        as the remaining part of $\partial \O$. 
\end{enumerate}
As a consequence of the above characterization of the blow-up limits, in \cite[Lemma 5.2]{MazzoleniTerraciniVelichkov:GAFA}, it was shown that the following optimality condition 
\be\label{e:viso}
|\nabla |U||=\sqrt{\Lambda} \quad \mbox{on }\partial \O\,
\ee
for the vector of eigenfunctions $U=(u_1,\dots,u_k)$, holds in viscosity sense on $\partial\Omega$. Precisely, for every $x_0 \in \partial \O$ and $\phi \in C^\infty(\R^d)$, we have
\begin{enumerate}
    \item if $\phi^+$ touches $|U|$ from above at $x_0$, then $|\nabla \phi(x_0)|\geq \Lambda$;
    \item if $\phi^+$ touches $|U|$ from below at $x_0$, then $|\nabla \phi(x_0)|\leq \Lambda$.
\end{enumerate}
 \medskip
 
\subsubsection{Regularity of the regular part} We next discuss the regularity of $\mathrm{Reg}(\partial \O)$ in the case of the functional $F=\lambda_1+\dots+\lambda_k$ and for general non-degenerate functionals.\\

\noindent Case 1: $F(A):=\sum_{i=1}^k \lambda_i(A)$. This case  was studied in \cite{MazzoleniTerraciniVelichkov:GAFA}. The analysis starts from the non-degeneracy of the first eigenfunction of the optimal domain (for a proof we refer to \cite{KriventsovLin2018:nondegenerate}), which allows to show that all the blow-up limits $U_0$ contain a component, which is positive on the cone $\Omega_{U_0}$. In particular, this allows to characterize the regular part as the set of points at which the Lebesgue density of $\O_U$ is $1/2$ and the singular part as the set of points at which the Lebesgue density of $\O_U$ is between $1/2+\delta$ and $1-\delta$ for some dimensional constant $\delta>0$. This in particular allows to prove that $\mathrm{Reg}(\partial \O)$ is relatively open.  It also allows to show (by considerations based on the Weiss' monotonicity formula and contradiction arguments) that for every regular point $x_0 \in \mathrm{Reg}(\partial \O)$ there exists a neighborhood $\mathcal{N}_{x_0}$ where $\O$ is a Reifenberg flat (and thus an NTA) domain. Now the boundary Harnack principle on NTA domains provides the existence of non-negative functions $g_i\colon \mathrm{Reg}(\partial \O)\cap \mathcal{N}_{x_0}\to \R$, with $i=2,\dots,k$, such that  
$$
u_i = g_i u_1 \quad\mbox{in }\overline{\O}\cap \mathcal{N}_{x_0}.
$$
Thus, from \eqref{e:viso} we get that there is a positive function $g \in C^{0,\alpha}$ such that
$$
|\nabla u_1|= \sqrt{\Lambda}g \quad \mbox{on }\mathrm{Reg}(\partial \O)\cap \mathcal{N}_{x_0},
$$
holds in a viscosity sense and so the $C^{1,\alpha}$-regularity simply follows from the $\eps$-regularity theorem of De Silva \cite{DeSilva:FreeBdRegularityOnePhase}. Finally, the smoothness of the regular part of $\partial \O$ is deduced by improving the regularity of $g$ via higher-order boundary Harnack inequalities (see \cite{HOBHDeSilva,HOBHdeg}) and then by relying on the results by Kinderlehrer and Nirenberg \cite{KinderlehrerNirenberg1977:AnalyticFreeBd}.\\

\noindent Case 2: General non-degenerate functionals. When $F$ is differentiable with respect to domain variations it is possible to show that the following boundary condition holds in a viscosity sense
$$
\sum_{i=1}^k \partial_{\lambda_i}F(\O) |\nabla u_i|^2 = \sqrt{\Lambda}\qquad \mbox{on }\partial \O\cap D.
$$
Therefore, the vector $\overline{U}:=(\sqrt{\partial_{\lambda_1}F(\O)}u_1,\dots,\sqrt{\partial_{\lambda_k}F(\O)}u_k)$ satisfies \eqref{e:viso} in a viscosity sense. If $F$ is just non-degenerate, then the terms $\partial_{\lambda_i}F(\O) $ are replaced by some non-negative coefficients $\xi_i$, which are obtained via an approximation argument similar to the one in \cite{RamosTavaresTerracini}. We postpone the discussion on the proof since it coincides with the one for the vectorial one-phase Bernoulli problem. The regularity of the regular part now follows as in the previous case. 

\subsubsection{Estimates on the dimension of the singular part} Finally, we briefly recall the few known results about the singular set $\mathrm{Sing}(\partial \O)$.
\begin{enumerate}
    \item For the functional $F(A):=\sum_{i=1}^k \lambda_i(A)$, we know that blow-up limits are minimizers of the Alt-Caffarelli functional. Therefore, the possible study of singular minimizing cones for the one-phase Bernoulli problem implies that $\mathrm{Sing}(\partial \O)$ is a closed set of Hausdorff dimension at most $d-d^\ast$, where $d^* \in \{5,6,7\}$ is the critical dimension for minimizers of the one-phase problem. \medskip
    \item For general non-degenerate functionals $F$, the blow-up limits are just stationary points of the Alt-Caffarelli functional (thanks to the fact that Property E holds in this case). Therefore, since singular critical cones can occur in any dimension greater than or equal to three, we only get that $\mathrm{Sing}(\partial \O)$ is a closed set of Hausdorff dimension at most $d-3$. See \cref{e:open-singular} for a discussion on the topic.  
\end{enumerate}

\subsection{Regularity of the optimal shapes for $\lambda_k(A)$ with $k>1$} 
Given constants $k\ge 2$ and $\Lambda>0$, and a fixed open set $D\subset \R^d$, consider the following the shape optimization problem
\begin{equation}\label{e:shape-opt-pb-for-lambda-k-in-D}
\min\Big\{\lambda_k(A) + \Lambda |A|:\ A\subset D \Big\},
\end{equation}
where $A$ varies in the class of quasi-open sets or in the class of Lebesgue measurable sets.
The main difference between the degenerate problem \eqref{e:shape-opt-pb-for-lambda-k-in-D} and the non-degenerate \eqref{e:sum-R-d-penalized} come from the variational characterization of $\lambda_k(A)$ and of $\sum_{j=1}^k\lambda_j(A)$ on a fixed open (or quasi-open, or just Lebesgue measurable) domain $A\subset \R^d$ of finite Lebesgue measure. Precisely, we have:
\begin{equation}\label{e:vectorial-reformulation-lambda-k-2}
\lambda_k(A)=\min\left\{\max_{j=1,\dots,k}\int_{A}|\nabla v_j|^2\,dx \ \colon \ v_j \in H^1_0(A),\, \int_{\R^d}v_iv_j\,dx = \delta_{ij}\right\}.
\end{equation}
In particular, the functional from \eqref{e:vectorial-reformulation-sum-lambda-k-2} 
\[
H^1_0(A;\R^k)\ni v=(v_1,\dots, v_k)\mapsto \sum_{j=1}^k\int_{A}|\nabla v_j|^2\,dx
\]
is smooth and quadratic, while the functional from \eqref{e:vectorial-reformulation-lambda-k-2}
\[
H^1_0(A;\R^k)\ni v=(v_1,\dots, v_k)\mapsto \max_{j=1,\dots,k}\int_{A}|\nabla v_j|^2\,dx
\]
 is not even differentiable with respect to outer variations. This makes the analysis of the optimal sets considerably harder. For instance, in the case of \eqref{e:sum-R-d-penalized} the Lipschitz continuity of the eigenfunctions on an optimal set follows essentially from the classical results of Alt-Caffarelli \cite{AltCaffarelli:OnePhaseFreeBd} and Alt-Caffarelli-Friedman \cite{AltCaffarelliFriedman1984:TwoPhaseBernoulli}, while establishing the Lipschitz continuity for eigenfunctions on minimizers of \eqref{e:shape-opt-pb-for-lambda-k-in-D} requires a  suitable approximation procedure and uniform Lipschitz bounds for minimizers of a sequence of functionals with degenerating coefficients (see \cite{BucurMazzoleniPratelliVelichkov:ARMA}).\medskip 

\noindent The singular nature of the functional in \eqref{e:shape-opt-pb-for-lambda-k-in-D} also reflects in the local structure of the optimal sets: the minimizers of \eqref{e:sum-R-d-penalized} admit exterior density estimates at all points on the free boundary, while the optimal sets for \eqref{e:vectorial-reformulation-lambda-k-2} can have boundary points with Lebesgue density $1$.\medskip

\noindent In the class of degenerate problems \eqref{e:shape-opt-pb-for-lambda-k-in-D} the case $k=2$ is special. When $k=2$ and $D=\R^d$, the minimizers are always the given by the union of two disjoint balls of the same radius. For $k=2$, $D$ bounded and $\Lambda>0$ small, minimizers are non-trivial sets touching the fixed boundary $\partial D$, while in the interior of $D$ they present cusp-like boundary singularities just as in the general case $k\ge 2$.

\subsubsection{Regularity of the optimal shapes for $\lambda_2$}

The regularity of the optimal sets for \eqref{e:shape-opt-pb-for-lambda-k-in-D} is particularly well-understood in the case $k=2$. Here, in order to deal with minimizers with non-trivial energy, we assume that the domain $D$ is a bounded open set. The full regularity of the optimal sets is obtained in \cite{MazzoleniTreyVelichkov:AnnIHP} (see  \cref{thm:MazzoleniTreyVelichkov} below) in domains $D$ with $C^{1,\alpha}$ boundary; in the case of a general bounded open set $D$ the regularity of the optimal domain is known only in the interior of $D$. 

\begin{theorem}[\cite{MazzoleniTreyVelichkov:AnnIHP}]\label{thm:MazzoleniTreyVelichkov}
Let $D\subset \R^d$ be an open  bounded set of class $C^{1,\beta}$, for some $\beta>0$, and let $\Lambda>0$ be a given constant. 
Let $\Omega\subset D$ be a minimizer of \eqref{e:shape-opt-pb-for-lambda-k-in-D}.
Then, there are two disjoint open sets $\Omega_+$ and $\Omega_-$ such that 
$$\lambda_2(\Omega_+\cup\Omega_-)=\lambda_2(\Omega)\qquad \text{and}\qquad |\,\Omega\,\setminus\,(\Omega_+\cup\Omega_-)|=0.$$
Moreover, each of the boundaries $\partial \Omega_\pm$ can be decomposed as
$$\partial \Omega_\pm=\text{\rm Reg}\,(\partial \Omega_\pm)\cup \text{\rm Sing}\,(\partial \Omega_\pm),$$
where 
$$\text{\rm Reg}\,(\partial \Omega_+)\cap \text{\rm Sing}\,(\partial \Omega_+)=\text{\rm Reg}\,(\partial \Omega_-)\cap \text{\rm Sing}\,(\partial \Omega_-)=\text{\rm Sing}\,(\partial \Omega_+)\cap \text{\rm Sing}\,(\partial \Omega_-)=\emptyset$$
and where the following holds:
	\begin{enumerate}[\quad\rm(i)]
		\item The regular set $\text{\rm Reg}(\partial \Omega_\pm)$ is an open subset of $\partial \Omega_\pm$, which is locally the graph of a $C^{1,\alpha}$ function, for some $\alpha>0$. Moreover, $\text{\rm Reg}(\partial \Omega_\pm)$ contains both the two-phase free boundary $\partial\Omega_+\cap\partial\Omega_-$ and the contact sets with the boundary of the box: $\partial\Omega_+\cap\partial D$ and $\ \partial\Omega_-\cap\partial D$.\medskip
		\item The singular set $\text{\rm Sing}(\partial \Omega_\pm)$ is a closed subset of $\partial \Omega_\pm$ and contains only one-phase points. Moreover, if $d^*\in\{5,6,7\}$ is the critical dimension for the minimizers of the one-phase Alt-Caffarelli's functional, then: 
		\begin{itemize}
			\item if $d<d^*$, then the singular set is empty,
			\item if $d=d^*$, then the singular set consists of a finite number of points,
			\item if $d>d^*$, then the singular set has Hausdorff dimension at most $d-d^*$.
		\end{itemize} 
	\end{enumerate}
\end{theorem}

We notice that the conclusion of the above theorem is not  that any optimal set $\Omega$ is regular, but that we can always choose a representative of $\Omega$ of the form $\Omega_+\cup\Omega_-$, where $\Omega_+$ and $\Omega_-$ are disjoint  sets, which are both regular. In fact, in general, there are optimal sets $\Omega$, which are not smooth. For instance, it is immediate to check that if $\Omega$ a minimizer and if $u_2\in H^1_0(\Omega)$ is the second eigenfunction of the Dirichlet Laplacian on $\Omega$, then the set $\{u_2>0\}\cup\{u_2<0\}$ is also a minimizer, as well as any set $\widetilde\Omega$ contained in $\Omega$ and containing $\{u_2>0\}\cap\{u_2<0\}$. One can use this observation to construct non smooth solutions by erasing portions of the nodal line $\partial\{u_2>0\}\cup\partial\{u_2<0\}$.\medskip

A key step in the proof of \cref{thm:MazzoleniTreyVelichkov} is the observation that the problem \eqref{e:shape-opt-pb-for-lambda-k-in-D} is equivalent to the variational problem 
\begin{equation*}\label{e:lambda-2-auxiliary-problem}
\min\Big\{\max\big\{\lambda_1(\Omega_1);\lambda_1(\Omega_2)\big\}+\Lambda|\Omega_1\cup\Omega_2|\ :\ \Omega_1\ \text{ and }\ \Omega_2\ \text{ are disjoint quasi-open subsets of }\ D\Big\},
\end{equation*}
which can also be written as the following two-phase free boundary problem whose minimizer is precisely the second eigenfunction $u_2$ on the optimal domain $\Omega$:
\begin{equation}\label{e:fb_infty}
\min\left\{J_\infty(v_+,v_-)+\Lambda\big|\{v\neq 0\}\big|\ :\ v\in H^1_0(D),\ \int_D v_+^2\,dx=\int_D v_-^2\,dx=1\right\},
\end{equation}
where 
\begin{equation}\label{e:J_infty}
J_\infty(v_+,v_-):=\max\left\{\int_D |\nabla v_+|^2\,dx\ ;\ \int_D |\nabla v_-|^2\,dx\right\}.
\end{equation}
The key point in the proof of \cref{thm:MazzoleniTreyVelichkov} is to select a Lipschitz continuous second eigenfunction $u\in H^1_0(\Omega)$ that satisfies the first order optimality conditions 
\begin{equation*}\label{e:opt_con_intro}
\begin{cases}
|\nabla u_+|=\alpha_+>0&\quad\text{on}\quad \partial\{u>0\}\setminus \partial\{u<0\}\cap D,\\
|\nabla u_-|=\alpha_->0&\quad\text{on}\quad \partial\{u<0\}\setminus \partial\{u>0\}\cap D,\\
|\nabla u_\pm|\ge \alpha_\pm\quad\text{and}\quad |\nabla u_+|^2-|\nabla u_-|^2=\alpha_+^2-\alpha_-^2&\quad\text{on}\quad \partial\{u>0\}\cap \partial\{u<0\}\cap D.
\end{cases}
\end{equation*}
in viscosity sense for some strictly positive constants $\alpha_+$ and $\alpha_-$. The conclusion then follows by the two-phase theorem \cite{DePhilippisSpolaorVelichkov:Inventiones} and by the one-phase regularity theory (see \cite{Velichkov:BookRegularityOnePhaseFreeBd}). Since the functional $J_\infty$ from \eqref{e:J_infty} is not differentiable, in \cite[Section 6]{MazzoleniTreyVelichkov:AnnIHP} this is achieved by approximating $u_2$ with minimizers $v_p$, as $p\to+\infty$, of a family of smooth functionals of the form
\begin{equation*}\label{eq:mainp}
\min\left\{\Big(\big(\mathcal R(v_+)\big)^p+\big(\mathcal R(v_-)\big)^p\Big)^{1/p}+\int_{D}|u_2-v|^2+|\Omega_v| : v\in H^1_0(D)\right\},
\end{equation*}
where by $\mathcal R(\phi)$ we denote the Rayleigh quotient of a non-zero function $\phi\in H^1(\R^d)$.

\subsubsection{Regularity of the optimal shapes for $\lambda_k$ with $k\ge 3$}
The case $k\ge 3$ shares several key features with the case $k=2$ discussed in the previous subsection. First of all, the regularity is not achieved directly for the optimal sets $\Omega$ but for a suitably selected representative of $\Omega$. The key difficulty lies again in the singular nature of the functional $\lambda_k$. On the other hand, when $k\ge 3$ the associated free boundary problem is not a two-phase problem, but a vectorial one, which significantly complicates the analysis. The most complete regularity result up to this point was achieved by Kriventsov and Lin and is the following. 

\begin{theorem}[\cite{KriventsovLin2019:degenerate}]\label{thm:Kriventsov-Lin-lambda_k}
 Let $D$ be an open set in $\R^d$ and let $\Omega$ be a (quasi-open) minimizer of \eqref{e:shape-opt-pb-for-lambda-k-in-D}. Then, there is a representative $\widetilde\Omega$ of $\Omega$ for which the free boundary $\partial\widetilde\Omega\cap D$ can be decomposed as the disjoint union of a regular part $\mathrm{Reg}(\partial \widetilde\O)$, a one-phase singular part $\mathrm{Sing}_1(\partial \widetilde\O)$, and a two-phase singular set $\mathrm{Sing}_2(\partial \widetilde\O)$,
   $$
   \partial \widetilde\O\cap D = \mathrm{Reg}(\partial \widetilde\O) \cup \mathrm{Sing}_1(\partial \widetilde\O)\cup \mathrm{Sing}_2(\partial \widetilde\O),
   $$
   where
   \begin{itemize}
        \item[(i)] $\mathrm{Reg}(\partial \widetilde\O)$ is a relatively open subset of $\partial\widetilde\Omega\cap D$ and a smooth manifold of dimension $(d-1)$. 
        \item[(ii)] $\mathrm{Sing}_1(\partial \widetilde\O)$ is a relatively closed subset of $\partial \widetilde\O \cap D$ that consists of all $x_0\in \partial\widetilde\Omega\cap D\setminus \mathrm{Reg}(\partial \widetilde\O)$ such that
        $$\liminf_{r\to0}\frac{|B_r(x_0)\cap\widetilde\Omega|}{|B_r|}<1.$$ 
        Moreover, if $d=2$, then $\text{\rm Sing}_1(\partial \widetilde\Omega)$ is empty, and if $d\ge 3$, then the Hausdorff dimension of $\text{\rm Sing}_1(\partial \widetilde\Omega)$ is at most $d-3$.
        \item[(iii)] $\mathrm{Sing}_2(\partial \widetilde\O)$ is a relatively closed subset of $\partial \widetilde\O \cap D$ that consists of all boundary points $x_0\in \partial \widetilde\O \cap D$ at which the Lebesgue density of $\widetilde\Omega$ is $1$.        
\end{itemize}
\end{theorem}

The proof of \cref{thm:Kriventsov-Lin-lambda_k} relies on the following family of non-degenerate functionals $F_p$ that approximate $F(\lambda_1,\dots,\lambda_k)=\lambda_k$ as $p \to \infty$ (see \cite{BucurMazzoleniPratelliVelichkov:ARMA,RamosTavaresTerracini} for similar constructions)
$$
F_p(\lambda_1(A),\dots,\lambda_k(A)) := p \int_0^{1/p}(\tau_{k,p}+t)\,dt + \frac{1}{p}\sum_{i=1}^k (k+1-i)\lambda_i(A),
$$
where, for all $i=1,\dots,k$, we set 
$
\tau_{i,p}(A):=\left(\sum_{j=1}^i \lambda_j(A)^p\right)^{1/p}.
$

Given $p>1$, $s>0$, and a minimizer $\Omega$ of 
\eqref{e:shape-opt-pb-for-lambda-k-in-D}, we consider the minimization problem
\begin{equation}\label{e:shape-opt-prob-aux-F-p}
\min\{F_p(A) + s\delta(A,\Omega) + \Lambda |A|:\ A\subset D\},
\end{equation}
where $\delta(A,\Omega)$ is a fidelity term of the form 
$$
\delta(A,\Omega):= \int_A \max\{1,\mathrm{dist}(x,\O)\} \,dx +\int_{A^c}\max\{1,\mathrm{dist}(x,\O^c)\}\,dx  + \phi(|\O|-|A|),
$$
where $\phi \in C^\infty(\R)$ is such that $\phi(0)=\phi'(0)=0, \phi(t)>0$ if $t\neq 0$, $|\phi'|\leq 1/2$. \smallskip

Let now $\Omega_p$ be a minimizer of \eqref{e:shape-opt-prob-aux-F-p}. The choice of $F_p$ assures that the eigenvalues $(\lambda_i(\Omega_p))_{i=1}^k$ are simple and that the sets $\O_p$ satisfy Theorem \ref{t:subsolution} and Theorem \ref{t:supersolution} (see \cite{KriventsovLin2019:degenerate}).
Computing the first variations of the functional $F_p$ along vector fields compactly supported in $D$, we get that the following condition is satisfied (in sense of domain variations) on the free boundary of $\Omega_p$:
$$
\sum_{i=1}^k \partial_{\lambda_i} F_p(\O_p) |\nabla u_{i,p}|^2 = \Lambda + s \max\{1,\mathrm{dist}(x,\O)\} - s \max\{1,\mathrm{dist}(x,\O^c)\} + s\phi'(|\O|-|\O_p|),
$$
where 
$U_p=(u_{i,p},\dots,u_{k,p})$ is the vector of the first $k$ normalized eigenfunctions on $\O_p$. \medskip

The convergence of $\Omega_p$ to a minimizer $\widetilde \Omega$ of \eqref{e:shape-opt-pb-for-lambda-k-in-D} was proved in \cite{KriventsovLin2019:degenerate}.
The fidelity term $\delta(A,\Omega)$ assures that $\delta(\widetilde\Omega,\Omega)=0$, so $\widetilde\Omega$ is a representative of $\Omega$. Moreover, the sequence $\partial\Omega_p$ converges in the Hausdorff distance to a closed set $Z$, which contains $\partial\widetilde\Omega$ and can be decomposed as follows:
$$
Z:= \mathrm{Reg}(\partial \widetilde\O) \cup \mathrm{Sing}_1(Z) \cup \mathrm{Sing}_2(Z),
$$
where $\mathrm{Reg}$, $\mathrm{Sing}_1$ and $\mathrm{Sing}_2$ have the following properties:
\begin{enumerate}[label=(\roman*)]
    \item $\mathrm{Reg}(\partial \widetilde\O)$ is the reduced boundary $\partial^* \widetilde\O$ of $\partial \widetilde\O$, i.e., the set of points of Lebesgue density $1/2$ for which the blow-up limit of $\Omega$ is an half-space;
    \item $\mathrm{Sing}_1(Z)$ is the set of one-phase type singular points, that is those point of Lebesgue density $\gamma \in [1/2+\delta,1)$, for some $\delta>0$. This set coincides with the essential boundary of $\widetilde\O$ minus the reduced boundary;
    \item $\mathrm{Sing}_2(Z)$ is the set of point of Lebesgue density $1$.
\end{enumerate}
Afterwards, being $\partial\widetilde\O\subset Z$, we can replace $\mathrm{Sing}_i(Z)$ with $\mathrm{Sing}_i(\partial \O)$. 
At this stage, one can observe one of the main differences between the degenerate and non-degenerate cases: in the degenerate setting, the existence of two-phase singular points cannot be ruled out a priori, since the exterior density estimate \eqref{e:exterior-density} does not hold in general.\medskip

Thanks to the fact that the functionals $F_p$ are non-degenerate and thanks to the regularity theory for non-degenerate functionals, we know that the each of the sets $\Omega_p$ has the regularity described in \cref{t:main-reg-sum}. On the other hand, we cannot expect that the free boundary regularity of $\Omega_p$ passes to the limit; for instance, in 2D the free boundaries $\partial\Omega_p$ are smooth, while we know (see the previous section) that the limit set $\widetilde\Omega$ might have cusps. This is a common situation when dealing with higher eigenvalues, see for instance \cite{RamosTavaresTerracini,KriventsovLin2019:degenerate,MazzoleniTreyVelichkov:AnnIHP}; the key idea in all these papers is to pass to the limit the main tools needed for proving the regularity of the free boundary (first variation formulas and optimality conditions) and then work directly with $\widetilde\Omega$. \medskip

Following this general strategy, in \cite{KriventsovLin2019:degenerate} it was proved that there is a (non-empty) collection of Lipschitz continuous eigenfunctions $u_i\in H^1_0(\widetilde \Omega\setminus Z)$, $i\in \mathcal I\subset\{1,\dots,k\}$, such that:
\begin{itemize}
\item for all $i\in\mathcal I$, $\displaystyle\xi_i:=\lim_{p\to+\infty}\partial_{\lambda_i}F_p(\Omega_p)$ is finite and strictly positive;
\item $\displaystyle\lim_{p\to+\infty}\lambda_i(\O_p)= \lambda_{\max\{j\in \mathcal{I}: \xi_i=\xi_j\}}(\O)$ for all $i\in\mathcal I$;
\item the following Weiss-type formula is monotone non-decreasing for all boundary points $x_0$   
        $$
        W(x_0,r):= \frac{1}{r^d} \sum_{i\in\mathcal{I}} \int_{B_r(x_0)} \xi_i|\nabla u_i|^2 \,dx + \frac{1}{r^d}|\O\cap B_r(x_0)| - \frac{1}{r^{d+1}}\int_{\partial B_r(x_0)}\sum_{i \in \mathcal{I}}\xi_i u_i^2\,d\mathcal{H}^{n-1}.
        $$
\item the following boundary condition holds in viscosity sense (see \cite[Section 7]{KriventsovLin2019:degenerate}):
\be\label{e:fb-degenerate}
\sum_{i\in \mathcal{I}} \xi_i |\nabla u_{i}|^2 = \Lambda \qquad \mbox{on}\quad \mathrm{Reg}(\partial \O) \cup \mathrm{Sing}_1(Z).
\ee        
\end{itemize}
Thus, from the perspective of $\eps$-regularity for free boundaries, the problem becomes equivalent to studying $\eps$-regularity for free boundaries arising from the vectorial Bernoulli functional \eqref{e:vectorial-bernoulli-functional}, without imposing any sign conditions on the components of the minimizing vector $U$. We defer the discussion about the vectorial $\eps$-regularity theory to  \cref{s:epsilon}. The estimate on the Hausdorff dimension of $\mathrm{Sing}_1(\partial \O)$ follows from the classical Federer's dimension reduction argument and the fact that the blow-up limits of $U$ are stationary and $1$-homogeneous. Finally, for what concerns the two-phase singular set, we notice that the $(d-1)$-rectifiability of $\mathrm{Sing}_2(Z)$ can be obtained by the Naber-Valtorta's theory (see \cref{s:branching}). 
  
\section{The vectorial Bernoulli problem}\label{s:vectorial-bernoulli}
In this section we discuss the regularity theory for local minimizers of \eqref{e:vectorial-bernoulli-functional}.
Thus, given a smooth open domain $D\subset \R^d$, a constant $\Lambda >0$ and a boundary datum $G:=(g_1,\dots,g_k)\in H^{1/2}(D;\R^k)$, we consider the minimization problem
\be\label{e:minproblem-vectorial}\min\left\{\mathcal{J}(V,D): V\in H^1(D;\R^k),\, V=G \,\mbox{ on }\partial D\right\},
\ee
where $\mathcal{J}(V,D)$ is defined as in \eqref{e:vectorial-bernoulli-functional}. Then, the main regularity result is the following.
\begin{theorem}[\cite{SpolaorVelichkov:CPAM,MazzoleniTerraciniVelichkov:AnalPDE}]\label{thm:MazzoleniTerraciniVelichkov2}
Every minimizer $U\in H^1(D;\R^k)$ of \eqref{e:minproblem-vectorial} is Lipschitz continuous
in $D$ and its positivity set $\O_U:=\{|U|>0\}$ has locally finite perimeter in $D$. Moreover, the free boundary $\partial \O_U$ can be decomposed as the disjoint union of three parts
   $$
   \partial \O_U\cap D = \mathrm{Reg}(\partial \O_U) \cup \mathrm{Sing}_1(\partial \O_U)\cup \mathrm{Sing}_2(\partial \O_U),
   $$
   with the following properties:
   \begin{itemize}
    \item[(i)] the regular part $\mathrm{Reg}(\partial \O_U)$ is a smooth manifold of dimension $(d-1)$ and
    $$
    |\nabla |U||=\sqrt{\Lambda} \quad \mbox{on}\quad \mathrm{Reg}(\partial \O_U);
    $$
    \item[(ii)] the one-phase singular part $\mathrm{Sing}_1(\partial \O_U)$ is a closed subset of $\partial \O_U \cap D$. Moreover, 
    \begin{itemize}
\item[(1)] if $d<d^*$, then $\text{\rm Sing}_1(\partial \Omega_U)$ is empty;
\item[(2)] if $d\ge d^*$, then the Hausdorff dimension of $\text{\rm Sing}_1(\partial \Omega_U)$ is at most $d-d^*$,
\end{itemize}
where the dimensional threshold $d^\ast\in\{5,6,7\}$ is the first dimension in which homogeneous minimizers of the Alt-Caffarelli functional may exhibit isolated singularities. 
        \item[(iii)] the two-phase singular set $\mathrm{Sing}_2(\partial \O_U)$ is a closed and of locally finite $(d-1)$-Hausdorff measure. It consists of only those points at which the Lebesgue density of $\partial \O_U$ is $1$.        
\end{itemize}
\end{theorem}
\subsection{Qualitative properties of minimizers} 
First, we notice that every component $u_i$ of a minimizer $U = (u_1, \dots , u_k)$ is an almost-minimizer of the Dirichlet energy, in the sense that 
$$
\int_{B_r(x_0)}|\nabla u_i|^2\,dx \leq \int_{B_r(x_0)}|\nabla \widetilde{u}_i|^2\,dx  + \Lambda |B_r|,\quad \mbox{for every }\widetilde{u}\in u+H^1_0(B_r(x_0)),
$$
for every $B_r(x_0)\subset D$ and every $\widetilde{u}\in u+H^1_0(B_r(x_0))$. This almost minimality implies that 
$U = (u_1, \dots , u_k)$ satisfies the following properties:
\begin{enumerate}[label=(\roman*)]
    \item {\it Lipschitz continuity \cite{BucurMazzoleniPratelliVelichkov:ARMA,Trey:COCV}.} $U$ is Lipschitz continuous in $D$.
    \item {\it Non-degeneracy \cite{KriventsovLin2019:degenerate,Trey:IHP}.} There exist $C_d,r_d$ positive, such that 
$$
\norm{U}{L^\infty(B_r(x_0))}\geq C_d r,\quad\mbox{for every}\quad x_0 \in \partial \O_U\quad\mbox{and}\quad r<r_d.
$$
\item {\it Interior density estimate \cite{MazzoleniTerraciniVelichkov:GAFA}.} There are constant $\eps_d, r_d$ positive such that 
$$
|\O_u\cap B_r(x_0)|\geq \eps_d |B_r|,\quad\mbox{for every}\quad x_0 \in \partial \O_U\quad\mbox{and}\quad r<r_d.
$$
\item {\it Perimeter of the positivity set \cite{MazzoleniTerraciniVelichkov:AnalPDE}.} $\O_u$ has locally finite perimeter in $D$.
\end{enumerate}
We notice that, due to the presence of two-phase singular points, we cannot hope to have an exterior density estimate of the form \eqref{e:exterior-density}. \medskip

\subsection{Monotonicity formula and blow-up analysis} A direct computation (see \cite{MazzoleniTerraciniVelichkov:GAFA,MazzoleniTerraciniVelichkov:AnalPDE}) gives that, for every $x_0\in \partial \O_U\cap D$, the Weiss' boundary adjusted energy
$$
W(x_0,r):=\frac{1}{r^d}J(U,B_{r}(x_0)) - \frac{1}{r^{d+1}}\int_{\partial B_r(x_0)}|U|^2\,d\mathcal{H}^{n-1} 
$$
is non-decreasing in $r$ (for $r$ small enough) and is constant if and only if $U$ is $1$-homogeneous with respect to $x_0$. Next we consider the $1$-homogeneous rescalings $U_{r,x_0}(x)=\frac1rU(x_0+rx)$ and we notice that, up to a subsequence, $U_{r_n,x_0}$ converges locally uniformly (uniformly on every compact subset of $\R^d$) to a non-trivial blow-up limit $$U_0:\R^d\to\R,$$ 
which is a Lipschitz continuous and $1$-homogeneous function that minimizes \eqref{e:vectorial-bernoulli-functional} in every compact subset of $\R^d$ (we will say that $U_0$ is a {\it global minimizer}). Moreover, the Lebesgue density of $\Omega_U$ at $x_0$ and the Lebesgue density of $\Omega_{U_0}$ at $0$ coincide. In particular, even if the blow-up $U_0$ might a priori depend on the sequence $r_n$, the Lebesgue density of $\Omega_{U_0}$ (at $0$) depends only on the point $x_0$. We distinguish between the following cases:
\begin{enumerate}
\item[(i)] {\it The Lebesgue density of $\O_U$ at $x_0$ is strictly smaller than $1$.} Then, there exists a unitary vector $\xi \in \R^k$ such that every blow-up $U_0$ of $U$ at $x_0$ is of the form
$$
U_0(x)=u_0(x)\xi\qquad \mbox{in }\R^d,
$$
for some scalar function $u_0$, which is a $1$-homogeneous minimizer of the Alt-Caffarelli functional. Moreover,\smallskip 
\begin{enumerate}
    \item if the Lebesgue density of $\O_U$ at $x_0$ is $1/2$, then 
    $$
    u_0(x) = \sqrt{\Lambda}(x\cdot \nu)^+,\quad\mbox{for some}\quad \nu\in\R^d,|\nu|=1;
    $$
    in this case we will say that $x_0$ is a regular point and we will write $x_0\in \mathrm{Reg}(\partial \O_U)$;\smallskip
    \item if the Lebesgue density of $\O_U$ at $x_0$ is in the open interval $(1/2,1)$, then the origin is a singular point for the one-phase homogeneous global minimizer $u_0$. In this case we will say that $x_0$ is a {\it one-phase singular point} and we will write $x_0\in\mathrm{Sing}_1(\partial \O_U)$.\medskip
\end{enumerate}
    \item[(ii)] {\it The Lebesgue density of $\O_U$ at $x_0$ is $1$.} Then, every blow-up limit $U_0$, of $U$ at $x_0$, is a linear function: $U_0(x) = Ax$ for some $A \in \R^{d\times k}$. In this case we will say that $x_0$ is a two-phase singular point; we will denote the set of these points by $\mathrm{Sing}_2(\partial \O_U)$.
\end{enumerate}
\subsection{The $\eps$-regularity theory}\label{s:epsilon} Up to this point, we know that the regular part of the free boundary $\mathrm{Reg}(\partial \O_U)$ coincides with the points of density $\sfrac12$ and can defined as the collection of points in which, up to a subsequence, the blow-ups converge to the limit
$$
U_0(x) = \sqrt{\Lambda}(x\cdot \nu)^+\xi\quad\mbox{for some}\quad \nu\in\R^d,|\nu|=1,
$$
and $\xi \in \R^k$ is a unitary vector. Since the free boundary of the blow-up limit is a hyperplane 
$$\O_{U_0} = \{x\in \R^d\colon x\cdot \nu >0\}\quad\text{and}\quad\partial \O_{U_0} = \{x\in \R^d\colon x\cdot \nu =0\},$$ 
it is natural to call such solution \emph{half-plane solutions} or \emph{flat solutions}.\\
There are several alternative approaches to the study of the regular points. In \cite{CaffarelliShahgholianYeressian2018:vectorial,MazzoleniTerraciniVelichkov:GAFA}, the problem was studied under the assumption that at least one of the components, say $u_1$, is non-negative in $D$. In both \cite{CaffarelliShahgholianYeressian2018:vectorial,MazzoleniTerraciniVelichkov:GAFA} this information was used to show that $\Omega_U=\Omega_{u_1}$ and that $u_1$ is a solution of a one-phase free boundary problem of the form 
\begin{equation}\label{e:one-phase-free-boundary-pb-reduction}
\Delta u_1=0\quad\text{in}\quad\Omega_{u_1}\ ,\qquad |\nabla u_1| =g\quad\text{in}\quad\partial\Omega_{u_1}\cap D,
\end{equation}
for a function $g$ which is H\"older continuous and such that $g\ge c>0$ on $\partial\Omega_{u_1}\cap D$. The regularity of $\rm Reg(\partial\Omega_U)$ is then obtained from the De Silva's result \cite{DeSilva:FreeBdRegularityOnePhase}. In both \cite{CaffarelliShahgholianYeressian2018:vectorial,MazzoleniTerraciniVelichkov:GAFA}, the validity of \eqref{e:one-phase-free-boundary-pb-reduction} is obtained by showing that the harmonic functions on the domain $\Omega_U$ satisfy the Boundary Harnack Principle. In \cite{MazzoleniTerraciniVelichkov:GAFA} this was done by first proving that $\Omega_U$ is an NTA domain (with arguments based on the Weiss' monotonicity formula) and then applying the theory on NTA domains. Later, in \cite{MazzoleniTerraciniVelichkov:AnalPDE}, it was shown that there always exists a component of the minimizing vector $U$ that does not change sign in a neighborhood of $\mathrm{Reg}(\partial \O_U)$, which competed the proof in higher dimension. Finally, we notice that the validity of the Boundary Harnack Principle on $\Omega_U$ can also be obtained directly (see \cite{MaialeTortoneVelichkov:AnnSNS} and \cite{CaffarelliShahgholianYeressian2018:vectorial}).\medskip

Another, more direct, approach to the regularity of $\rm Reg(\partial\Omega_U)$ was developed in the works of Kriventsov and Lin \cite{KriventsovLin2018:nondegenerate,KriventsovLin2019:degenerate}, where the $C^{1,\alpha}$ regularity of $\rm Reg(\partial\Omega_U)$ is obtained directly for vector-valued functions $U$ satisfying a particular viscosity condition on the free boundary.  The free boundary condition in the vectorial setting can often be expressed in several different ways. We will next focus on the viscosity formulation \eqref{e:viso}, which was introduced in \cite{MazzoleniTerraciniVelichkov:GAFA}. 

\begin{definition}[Vectorial viscosity solutions]\label{d:viscos}
Let $U\colon D \to \mathbb{R}^k$ be a continuous vector-valued function. We say that $U$ is a viscosity solution to 
\be\label{VOP}
\Delta U = 0 \quad\mbox{in }\O_U\cap B_1,\qquad 
\mbox{and}\qquad 
|\nabla |U||=\sqrt{\Lambda}\quad\text{on }\partial \O_U\cap B_1\,,
\ee
if each component of $U$ is harmonic in $\Omega_{U}=\{|U|>0\}$
and if for every test function $\varphi \in C^2(D)$, the following holds:
\begin{enumerate}
\item[(i)] if $\varphi_+$ touches $|U|$ from below at some $x_0\in\partial\Omega_U\cap D$ and $|\nabla\varphi(x_0)|\neq 0$, then $|\nabla \varphi(x_0)| \le 1$;
\item[(ii)] if $\varphi_+$ touches $|U|$ from above at some $x_0\in\partial\Omega_U\cap D$, then $|\nabla \varphi(x_0)| \ge 1$.
\end{enumerate}
\end{definition}

An epsilon-regularity theory for vectorial viscosity solutions in the sense of \cref{d:viscos} was developed in \cite{DeSilvaTortone2020:ViscousVectorialBernoulli}. We notice that in \cite{DeSilvaTortone2020:ViscousVectorialBernoulli} the condition (i) was replaced by 
\begin{enumerate}
\item[(i')] if, for some unit direction $\xi \in \R^k$, $\varphi_+$ touches $(U\cdot \xi)$ from below at some $x_0\in\partial\Omega_U\cap D$ and $|\nabla\varphi(x_0)|\neq 0$, then $|\nabla \varphi(x_0)| \le 1$.
\end{enumerate}
The class of vector-valued functions satisfying the free boundary condition (i') contains the ones satisfying the free boundary condition (i) from \cref{d:viscos}. Thus, the theory developed in \cite{DeSilvaTortone2020:ViscousVectorialBernoulli} covers a more general class of solutions.

\begin{theorem}[\cite{DeSilvaTortone2020:ViscousVectorialBernoulli}]
Let $U$ be a viscosity solution to \eqref{VOP}, where the boundary condition is satisfied in the sense of (i') and (ii). Then, there exists a $\bar\eps >0$ universal such that if $U$ is $\bar \eps$-flat in $B_1$, that is 
\be \label{flat} 
|U(x) - \sqrt{\Lambda}(x\cdot \nu)^+ \xi| \leq \bar \eps \quad \text{in}\quad B_1,
\ee
for some unit directions $\nu \in \R^d, \xi\in \R^k$, and if\be\label{nondegenra}
 |U| \equiv 0 \quad \text{in}\quad\{x \in B_1\colon (x \cdot \nu)  < - \bar \eps\},
 \ee
then $\partial \O_U \cap B_{\sfrac12}$ is the graph of a $C^{1,\alpha}$-function in the direction $\nu$.\end{theorem}

Naturally, the key-lemma is the validity of an improvement of flatness
(see \cite[Section 8]{Velichkov:BookRegularityOnePhaseFreeBd} for a complete proof of how the improvement of flatness implies the $C^{1,\alpha}$-regularity of the free boundary): roughly speaking, if $0\in \partial \O_U$ and $U$ is $\eps$-flat in $B_1$, then there exists a universal scaling factor $r_0>0$ such that $U_{0,r_0}$ is $\eps/2$-flat in $B_1$ (here we use the notation for blow-up sequences \eqref{e:blow-up-sequence}). By iterating this argument, one can establish the existence of a tangent hyperplane at every point of $\mathrm{Reg}(\partial \O_U)$, and subsequently deduce the H\"{o}lder continuity of the normal vector.
\begin{lemma}[Improvement of Flatness, \cite{DeSilvaTortone2020:ViscousVectorialBernoulli}]\label{IMPF} Let $U$ be an $\eps$-flat viscosity solution to \eqref{VOP}, the boundary condition being satisfied in the sense of (i') and (ii). Let $0 \in \partial \O_U$. 
Assume that
\be\label{flat1}  
|U- \sqrt{\Lambda}(x_d)^+ \xi_1| \leq \eps \quad \text{in}\quad B_1,
\ee
where $\xi_1$ is the first element of the canonical basis in $\R^k$, and
\be\label{non_d1} 
|U| \equiv 0 \quad \text{in}\quad \{ x \in B_1 \colon \cap x_d < - \eps\}.
\ee
Then, there exist $r_0,\eps_0>0$ universal such that, if $\eps\in (0,\eps_0]$
\be\label{flat_imp}
|U - \sqrt{\Lambda}(x\cdot  \nu)^+ \xi| \leq \frac{\eps}{2}{r_0}\quad \text{in $B_{r_0}$},
\ee
and
\be\label{non_dimpr} |U| \equiv 0 \quad \text{in $\left\{x \in B_{r_0} \colon  (x \cdot \nu) < - \frac{\eps}{2}{r_0}\right\}$}, \ee
with $|\nu -e_d|\leq C\eps, |\xi-\xi_1| \leq C \eps$, for a universal constant $C>0.$
\end{lemma}
The main idea consists in tracking the improvement of $|U|$, rather than working component-wise as in \cite{KriventsovLin2019:degenerate,KriventsovLin2018:nondegenerate}. By working with $|U|$, some of the difficulties related to the no-sign assumption can be avoided. Indeed, the flatness assumption on the vector-valued function $U$ implies that one of its components is trapped between nearby translation of a one-plane solution,
while the remaining ones are small (in a quantitative way). Precisely, if 
$$
|U- \sqrt{\Lambda}(x_d)^+ \xi_1| \leq \eps \quad \text{in $B_1$},\qquad |U| \equiv 0 \quad \text{in }\left\{x \in B_{1} \colon  x_d < - \eps\right\}
$$
we get that 
\begin{enumerate}
\item for every $i=2,\ldots,m$, we have $|u^i| \leq C \eps(x_d+\eps)^+$ in $B_{3/4}$;\\[-0.5em]
\item $\sqrt{\Lambda}(x_d-\eps) \leq u_1 \leq |U| \leq \sqrt{\Lambda}(x_d+2\eps)^+$ in $B_1$.
\end{enumerate}
In particular, in the latter inequalities we have both $|U|$ and $u_1$ are scalar functions that inherit the flatness of $U$. Moreover, by \cref{d:viscos}, since 
$$
\begin{cases}
\Delta |U| \geq 0 &\mbox{in }\O_U\cap B_1\\
|\nabla |U||=\sqrt{\Lambda} &\mbox{on }\partial \O_U\cap B_1,
\end{cases}\quad
\begin{cases}
\Delta u_1 = 0 &\mbox{in }\O_U\cap B_1\\
|\nabla u_1|\leq \sqrt{\Lambda} &\mbox{on }\partial \O_U\cap B_1,
\end{cases}
$$
they are respectively a viscosity subsolution and supersolution of the (scalar) one-phase problem. Then, even if $u_1$ may change sign, we can show that either the flatness of $u_1$ is improved from below or the flatness of $|U|$ is improved from above (meanwhile, the other components become increasingly negligible as the flatness improves). 
By iterating this argument, following the general scheme of De Silva, one can deduce the compactness of the linearized sequences and, subsequently, regularity via a linearization of the form
$$
U = \sqrt{\Lambda}(x_d)^+ \xi_1 + \eps\tilde{U}_\eps, \quad |U| = \sqrt{\Lambda}(x_d)^+ + \eps V_\eps\qquad\mbox{in }\overline{\O_U}\cap B_1,
$$
where $V_\eps = (\tilde{U}_\eps \cdot \xi_1) + o(\eps^{3/2})$, as $\eps\to 0$. \medskip

\begin{remark}
The strategies from \cite{CaffarelliShahgholianYeressian2018:vectorial,MazzoleniTerraciniVelichkov:GAFA,MazzoleniTerraciniVelichkov:AnalPDE,KriventsovLin2018:nondegenerate,KriventsovLin2019:degenerate,DeSilvaTortone2020:ViscousVectorialBernoulli} are all based on the validity of the optimality condition in viscosity sense and can be applied not only to minimizers, but also to other classes of (stationary or viscosity) solutions. For minimizers, on the other hand, the regularity of the free boundaries without the sign assumption from \cite{CaffarelliShahgholianYeressian2018:vectorial,MazzoleniTerraciniVelichkov:GAFA,KriventsovLin2018:nondegenerate} was first obtained in \cite{SpolaorVelichkov:CPAM} via an epiperimetric inequality in dimension $d=2$. Whether this epiperimetric inequality holds in every dimension $d\ge 2$ is still an open question. 
\end{remark}
\subsection{The two-phase singular set $\mathrm{Sing}_2(\partial \O_U)$}\label{s:branching}
By the blow-up analysis, we already know that 
$$
x_0 \in \mathrm{Sing}_2(\partial \O_U) \,\,\longleftrightarrow\,\, x_0 \in \O_U^{(1)} \,\,\longleftrightarrow \,\,\begin{cases}
\text{Every blow-up $U_0$ at $x_0$ is a linear function},\\ 
\text{i.e. $U_0(x)=Ax$ for some matrix $A \in \R^{d\times k}$,}\end{cases}
$$
where by $\O_U^{(1)}$ we denote the set of boundary points at which $\Omega_U$ has Lebesgue density $1$.\medskip

Initially, in \cite[Lemma 4.2]{MazzoleniTerraciniVelichkov:GAFA} the authors showed that although the blow-up limit may be not unique, the rank of the matrix $A$ depends only on $x_0$ (for simplicity we denote with $\mathrm{Rk}(x_0)$ the rank of the blow-up limit at $x_0$). We notice that $1\le \mathrm{Rk}(x_0)\le \min\{k, d\}$. In particular:
\begin{enumerate}
    \item if $\mathrm{Rk}(x_0)=1$, then there is a unit vector $\nu \in \R^d$ such that the rows of the matrix $A$ are the vectors $\alpha_i \nu,\dots,\alpha_k \nu$, where $\alpha:=(\alpha_1,\dots,\alpha_k) \in \R^k$ satisfies
    $$
    |\alpha|^2 = \sum_{i=1}^k \alpha_k^2 \geq \Lambda.
    $$
    \item if $1<\mathrm{Rk}(x_0)\leq k$, the matrices $A := (a_{ij})_{ij}$ associated with blow-ups of $U$ at $x_0$ satisfy
$$
|A|^2 = \sum_{1\leq i,j\leq d}a_{ij}^2 \geq C_d\Lambda,
$$
    where $C_d$ is a dimensional constant.
\end{enumerate}
Therefore, we can decompose the two-phase singular set $\mathrm{Sing}_2(\partial \O_U)$ in terms of the rank
$$
\mathcal{S}_j:=\{x_0 \in \mathrm{Sing}_2(\partial \O_U): \mathrm{Rk}(x_0)=j\},\quad\mbox{with }j=1,\dots,k.
$$
Moreover, by \cite{MazzoleniTerraciniVelichkov:AnalPDE} with \cite[Theorem 1.2]{rectifiabilityvectorial}, we know that for every $j = 1,\dots, k$, the $j$-th stratum $\mathcal{S}_j$ is $(d - j)$-rectifiable and has locally finite $(d - j)$-dimensional Hausdorff measure.
\section{Free boundary systems and optimal domains for energy functionals}\label{s:vectorial-fg}
As already mentioned in Subsection \ref{s:subs-system}, there exists another remarkable class of shape optimization problems in which the presence of two state equations gives rise to a free boundary system. Therefore, let us recall the variational framework associated to the class of {\it integral shape functionals}. Consider the minimization problem associated to 
$$J(A):=\int_A j(u_A,x)\,dx\,,$$
where $A\in \mathcal{A}$ is an admissible domain (e.g., quasi-open or Lebesgue measurable sets $A\subset D$), the cost function $j:\R\times\R^d\to\R$ is fixed and the state function $u_A$ is the (weak) solution of the PDE
\begin{equation}\label{e:state-equation-open}
-\Delta u_A=f\quad \text{in}\quad A\,,\qquad u_A\in H^1_0(A)\,,
\end{equation}
where the force term $f:\R^d\to\R$ is also a prescribed function. As in the case of spectral functionals, it is natural to introduce the 
\begin{enumerate}
    \item[(i)] Penalized problem: given $\Lambda>0$ consider 
    \be\label{e:fg.penalized}\min\Big\{\int_A j(u_A,x)\,dx +\Lambda|A|\, :\ A\subset D\mbox{ quasi-open}\Big\};\ee
    \item[(ii)] Measure constrained problem: given $c>0$ consider 
    \be\label{e:fg.consrtrained}\min\Big\{\int_A j(u_A,x)\,dx\, :\ A\subset D \mbox{ quasi-open},\,|A|=c\Big\}.\ee
\end{enumerate}
For a general existence theory for these problems we refer to the book \cite{BucurButtazzoBook} and the references therein. Precisely, it is well-known that if $D\subset \R^d$ is a bounded open set and $j$ satisfies some suitable growth assumptions, then both the
penalized and the measure constrained problem admit optimal shapes
 $\Omega\subset D$, while the existence of open solutions was proved in \cite{ButtazzoMaialeVelichkov:Lincei}.

As already stressed in the introduction, it is convenient to introduce the so-called \emph{adjoint state function} $v_A$ as follows: for every open set $A\subset D$ we denote by $v_A$ the weak solution to 
\be\label{e:stateqv}
-\Delta v_A=g\quad \text{in}\quad A\,,\qquad v_A\in H^1_0(A)\,.
\ee
Then, by an integration by parts one can see that
$$\int_{D}gu_A\,dx=\int_D\nabla u_A\cdot\nabla v_A\,dx=\int_{D}fv_A\,dx\,,$$
which means that the two state variables $u_A$ and $v_A$ are interchangeable. Therefore, by an integration by parts, we the cost functional admits the following alternative formulation:  
$$
J(A) = \int_D \left(\nabla u_A \cdot \nabla v_A - g(x)u_A -f(x)v_A\right) \,dx
$$
for every open set $A \subset D$.

Recently, in the series of papers \cite{ButtazzoMaialeMazzoleniTortoneVelichkov:ARMA,MaialeTortoneVelichkov:RMI,MaialeTortoneVelichkov:AnnSNS} with Buttazzo, Maiale and Mazzoleni, we developed a regularity theory in the simplest case
\begin{equation}
j(u,x) := -g(x)u\,,\quad\text{with}\quad g\neq f.
\end{equation}
In particular, the results from \cite{ButtazzoMaialeMazzoleniTortoneVelichkov:ARMA, MaialeTortoneVelichkov:RMI,MaialeTortoneVelichkov:AnnSNS} apply to non-negative functions $f,g\colon D\to \R$ such that
\begin{enumerate}
    \item[(a)] $f,g \in L^\infty(D)\cap C^\infty(D)$;
    \item[(b)] there are constants $C_1, C_2 > 0$ such that
    $$
    0\leq C_1 g \leq f \leq C_2 g \quad\mbox{in }\overline{D}.
    $$
\end{enumerate}
Our main regularity result is the following.
\begin{theorem}[\cite{ButtazzoMaialeMazzoleniTortoneVelichkov:ARMA}]
Let $D\subset\R^d$ be a bounded open set, and suppose that the functions $f,g:D\to\R$ satisfy the assumptions (a) and (b). Then, both the problems \eqref{e:fg.penalized} and \eqref{e:fg.consrtrained} admit minimizers in the class of open sets. Moreover, in both cases, the free boundary of an optimal shape $\Omega$ can be decomposed as the disjoint union of regular and singular parts
   $$
   \partial \O\cap D = \mathrm{Reg}(\partial \O) \cup \mathrm{Sing}(\partial \O),
   $$
   where:
   \begin{itemize}
        \item[(i)] the regular part $\mathrm{Reg}(\partial \O)$ is a smooth manifold of dimension $(d-1)$. Moreover
        $$
        |\nabla u_\Omega||\nabla v_\Omega|=\Lambda \quad \mbox{on}\quad\mathrm{Reg}(\partial \O),
        $$
        where $u_\O$ and $v_\O$ are the state variables associated to the optimal shape $\O$;\smallskip
        \item[(ii)] the singular part $\mathrm{Sing}(\partial \O)$ is a closed subset of $\partial \O \cap D$. Moreover, there exists a dimensional threshold $d^*\in \{5,6,7\}$ such that the following properties holds
\begin{itemize}
\item[(1)] If $d<d^*$, then $\text{\rm Sing}(\partial \Omega)$ is empty.
\item[(2)] If $d\ge d^*$, then the Hausdorff dimension of $\text{\rm Sing}(\partial \Omega)$ is at most $d-d^*$. The dimensional threshold $d^\ast$ coincides with the first dimension in which homogeneous minimizers of the Alt-Caffarelli functional may exhibit isolated singularities. 
\end{itemize}
\end{itemize}
\end{theorem}
The previous theorem is presented in a simplified form compared to the version in \cite{ButtazzoMaialeMazzoleniTortoneVelichkov:ARMA} (see the discussion in \cite[Remark 1.4 and Remark 1.5]{ButtazzoMaialeMazzoleniTortoneVelichkov:ARMA}). Nevertheless, assumptions (a)–(b) are the most restrictive conditions employed in the analysis of the regularity of both the state variables and the free boundary. In particular, assumption (b) represents the main obstacle to extending the theory to more general integral cost functionals.\\

We stress that the special case in which $f$ and $g$ are proportional has been already studied in literature. Precisely, (see \cite[Subsection 1.1]{ButtazzoMaialeMazzoleniTortoneVelichkov:ARMA}) if $f$ and $g$ are such that
\begin{equation}\label{e:intro-proportional}
g=\frac{1}{2 C^2} f\quad\text{for some constant}\quad C>0\,,
\end{equation} it is possible to deduce that the state variables inherits the proportionality, and so the penalized and the constrained problem are respectively equivalent to the free boundary problems
$$
\min\bigg\{\int_D\Big(\frac12|\nabla u|^2-f(x)u+ C^2 \Lambda\ind_{\{u\neq 0\}}\Big)\,dx\ :\ u\in H^1_0(D)\bigg\}.
$$
and 
$$
\min\bigg\{\int_D\Big(\frac12|\nabla u|^2-f(x)u\Big)\,dx\ :\ u\in H^1_0(D), |\{u\neq 0\}|=c\bigg\}.
$$
Finally, we notice that when $f$ and $g$ are not proportional, the state function $u_\Omega$ of an optimal set $\Omega$ is not a minimizer of a Bernoulli-type  functional (neither a function of the state variables $u_\O, v_\O$). In particular, this implies that one can test the optimality of $u_\Omega$ only with functions $\widetilde u$ which are themselves state functions of some admissible $\widetilde \Omega$. This means that a function $\widetilde u$, that differs from $u$ only in a small ball $B_r$, cannot be used to test the optimality of $u_\Omega$ (truncations, harmonic replacements and radial extensions in small balls are not admitted), which makes most of the classical free boundary regularity results impossible to apply.\\

For the sake of simplicity, in this section we review the main results only for the penalized problem \eqref{e:fg.penalized}.
Indeed, by adapting the methodologies developed in \cite{Briancon2010,briancon1}, it is possible to extend the study of the state functions and the $\varepsilon$-regularity theorem to the case of minimizers subject to a measure constraint.
However, since the analysis of singular points depends on the specific structure of the problem, in \cref{s:singular.constrained} we highlight the main differences between the two approaches and present a versatile methodology for studying singular points in the measure-constrained case.

\subsection{Qualitative properties of the state variables} 
In order to deduce the optimal regularity and qualitative properties of the state variable of the state functions, it is possible to exploit again the notion of shape subsolution (inward minimality) and shape supersolution (outward minimality). 
Nevertheless, this is the point in which assumption (b) is used and plays a major role  in the proofs of the Lipschitz continuity and the non-degeneracy of the state function $u_\Omega$ (we notice that (b) is automatically satisfied when $f$ and $g$ are both bounded from above and from below by positive constants). We postpone to \cref{o:generalizedfg} the discussion on the role of the assumption (b) in future directions.\\

Given a set $A\subset\R^d$, and functions $f\in L^2(A)$ and $\varphi\in H^1(A)$, we set
$$
E_f(\varphi,A):=\frac12\int_A|\nabla \varphi|^2\,dx-\int_Af(x)\varphi\,dx\,.
$$
Then in \cite[Proposition 3.1]{ButtazzoMaialeMazzoleniTortoneVelichkov:ARMA} they showed if $\Omega\subset D$ is a shape optimizer of the penalized problem, then the state variable $u_\Omega$ satisfies the following properties.
\begin{enumerate}[\rm (i)]
\item Shape supersolution (outwards minimality). For every open set $\widetilde \Omega\subset D$ such that $\O\subset\widetilde\Omega$ we have
$$E_f(u_\Omega,D) + \frac{C_2}{2}\Lambda|\O|\le
E_f(\phi,D) + \frac{C_2}{2}\Lambda|\widetilde\Omega|\qquad\mbox{for every }\,\phi \in H^1_0(\widetilde \Omega);$$
\item  Shape subsolution (inwards minimality). For every open set $\omega\subset \Omega$ we have
$$
E_f(u_\Omega,D) + \frac{C_1}{2}\Lambda|\Omega|\le
E_f(\phi,D) + \frac{C_1}{2}\Lambda|\omega|\qquad\mbox{for every }\,\phi \in H^1_0(\omega)\,.
$$
\end{enumerate}
Then, the Lipschitz continuity follows by applying \cite[Theorem 3.3]{BucurMazzoleniPratelliVelichkov:ARMA}, while the non-degeneracy is a consequence of \cite[Lemma 4.4]{AltCaffarelli:OnePhaseFreeBd}, \cite[Lemma 3.3]{BucurVelichkov-multiphase}. Naturally, being the two state variable interchangeable, the same results hold for $v_\O$.
Starting from this observation, it is possible to deduce further properties of the optimal sets such as density estimates of the form \eqref{e:interior-density} and \eqref{e:exterior-density}. In light of the discussion of the previous section, it is natural to except that validity of an exterior density estimate will prevent the occurrence of two-phase singular points.

\subsection{Blow-up analysis}\label{s:triple}
Since both $u_\Omega$ and $v_\Omega$ are Lipschitz continuous and non-degenerate, we can perform a blow-up analysis, by following the strategy explained in the previous sections.
However, there are two main difficulties that we encounter when  following this path.

\subsubsection{Criticality and stability along inner variations} The first difficulty comes from the impossibility to make external perturbations of $u_\Omega$ and $v_\Omega$ (that is no perturbations of the form $\widetilde u=u_\Omega+\phi$ and $\widetilde v=v_\Omega+\psi$ are allowed). Indeed, in this context the only admissible variations are the one associated to domain variations obtained as follows. Given an optimal set $\O$ in $D$ and $\xi \in C^\infty_c(D;\R^d)$, consider the flow $\Phi_t:=\mathrm{Id}+t\xi$ generated by $\xi$ and set $\O_t:=\Phi_t(\O)$. Then, by the minimality of $\Omega$, we have
\begin{equation}\label{e:intro-stable-critical-point}
\delta J(\O)[\xi]:=\frac{\partial}{\partial t}\Big|_{t=0} J(\Omega_t)=0\qquad\text{and}\qquad \delta^2 J(\O)[\xi]:=\frac{\partial^2}{\partial t^2}\Big|_{t=0} J(\Omega_t)\ge0\,.
\end{equation}
We notice that domains $\Omega$ that satisfy these conditions for every smooth vector fields with compact support are said to be stable critical points of $J$.\medskip

\noindent\begin{minipage}{0.98\textwidth}
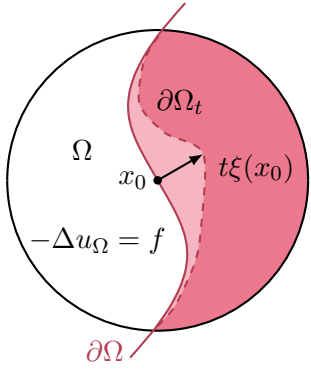
\begin{minipage}{0.33\textwidth}
      \begin{tikzpicture}[scale=2.0]
     \begin{scope}[rotate=30]
  \clip (0,0) ellipse (1cm and 1cm);
\filldraw[red2!30, domain=-1:1]  (0,-2) -- plot (\x^4,\x ) -- (2,0) -- (0,-2) -- cycle;
  \fill[red2!55] plot [smooth] coordinates {(-0.486,-0.84) (-0.316,-0.78)  (-0.062,-0.6) (0.09,-0.42)  (0.35,0) (0.2,0.25) (0.15,0.5) (0.316,0.73)  (0.488,0.83) (1,1) (2,0) (1,-2) (-1,-1)};
 \end{scope}
     \draw[black,thick] (1,0) arc (0:180:1cm);
    \draw[black,thick] (1,0) arc (0:-180:1cm);
  \draw[white] (-1.44,0.1);
  \draw[color=red2!55!gray,thick] (-0.93^4+0.2,-1) node[anchor=north west]{$\partial \Omega$};
  \draw[color=red2!55!gray, thick, domain=-0.93:0.93, rotate=30,smooth] plot
  (\x^4,\x );
  \draw[black] (0.15,0.52) node {$\partial \O_t$};
  \draw [color=red2!55!gray,thick,dashed,rotate=30] plot [smooth] coordinates {(-0.486,-0.84) (-0.316,-0.78)  (-0.062,-0.6) (0.09,-0.42)  (0.35,0) (0.2,0.25) (0.15,0.5) (0.316,0.73)  (0.488,0.83)};
   \filldraw[black] (0,0) circle (0.7pt) node[anchor= east] {$x_0$};
   \draw[black] (-0.4,-0.4) node {$-\Delta u_\O=f$};
   \draw[black,rotate=30] (0.33,-0.1) node[anchor=west] {$t \xi(x_0)$};
   \draw[-latex,black,thick,rotate=30] (0,0) -- (0.35,0);
    \draw[black] (-0.5,0.2) node {$\Omega$};
   \end{tikzpicture}
  \end{minipage}
     \quad 
  \begin{minipage}{0.66\textwidth}
  \vspace{-0.2cm}
  Now, the only information conserved along blow-up sequences is the one contained in the internal variations of $\Omega$ along smooth vector fields. We stress that, in order to compute $\delta J$ and $\delta^2J$ along the vector field $\xi$,
  we need to keep track of the state variables $u_{\O_t}, v_{\O_t}$ along the flow generated by $\xi$. Precisely, for every time $t$ we solve 
  $$
  \hspace{-0.1\linewidth}\begin{cases}
  -\Delta u_{\O_t} = f &\mbox{in }\Omega_t\\
  -\Delta v_{\O_t} = g &\mbox{in }\Omega_t\\
  u_{\O_t} = v_{\O_t}=0 & \mbox{on }\partial \Omega_t.
  \end{cases}
  $$
  \end{minipage}
  \end{minipage}
\vspace{0.2cm}\\
At this point, we only know that the $\O$ is an open set whose boundary admits a relatively open subset which is regular (in the sense specified in Subsection \ref{s:eps-fg}). Thus, in order to compute explicitly both the first and the second variation in \eqref{e:intro-stable-critical-point}, we need a local expansion of the state variables along the domain variations (see \cite[Section 2]{ButtazzoMaialeMazzoleniTortoneVelichkov:ARMA} for a detailed discussion on this topic). Roughly speaking, being $\O$ an open set, the maps $t \mapsto u_{\O_t}$ is not $C^2$-differentiable, whereas 
$$
t \mapsto (u_{\O_t}\circ \Phi_t\colon \O\to \R) \quad \text{is $C^2$-diferentiable}.
$$
Indeed, we can show the validity of a Taylor-type expansion (see \cite[Proposition 2.5]{ButtazzoMaialeMazzoleniTortoneVelichkov:ARMA}) 
$$
u_{\O_t}\circ \Phi_t  = u_\O + t(\delta u_\O) + t^2(\delta^2 u_\O) + o(t^2), \text{ as $t\to 0^+$}
$$
where 	\begin{align*}
	-\Delta (\delta u_\O)=\dive((\delta A)\nabla u_\O)+\delta f &\quad\text{in}\quad\Omega\,,
	\qquad\delta u_\O\in H^1_0(\O),\\
	-\Delta (\delta^2 u_\O)=\dive((\delta A)\nabla (\delta u_\O))+\dive((\delta^2 A)\nabla u_\O)+\delta^2 f &\quad\text{in}\quad\Omega\,,
	\qquad\delta^2 u_\O\in H^1_0(\O),
	\end{align*}
	the matrices $\delta A\in L^\infty(D;\R^{d\times d}),  \delta^2A\in L^\infty(D;\R^{d\times d})$ are given by
	\begin{align*}
	\begin{aligned}
	\delta A &:=-(\nabla \xi)^T-\nabla\xi+(\dive\xi)\text{\rm Id}\,,\\
	\delta^2 A &:= (\nabla \xi)^T\,(\nabla\xi)+\frac12\big(\nabla\xi\big)^2+\frac12((\nabla \xi)^T)^2-\frac12(\xi\cdot\nabla)\big[\nabla\xi+(\nabla \xi)^T\big]\\
	&\qquad-\big(\nabla\xi+(\nabla \xi)^T\big)\dive\xi+\text{\rm Id}\frac{(\dive\xi)^2+\xi\cdot\nabla(\dive\xi)}{2}\,,
	\end{aligned}
	\end{align*}
	while the variations $\delta f\in L^2(D)$ and $\delta^2f\in L^2(D)$ of the right-hand side $f$ are:
	\begin{align*}
	\begin{aligned}
	\delta f &:=\dive(f\xi)\,,\\
	\delta^2 f &:= \frac12 \xi \cdot (D^2f)\xi+\frac12\nabla f\cdot \nabla\xi[\xi]+f\frac{(\dive\xi)^2+\xi\cdot\nabla[\text{\rm div}\,\xi]}{2}+(\nabla f \cdot \xi)\dive\xi \,.
	\end{aligned}
	\end{align*}
Naturally, since the state variables $u_\O, v_\O$ are interchangeable, a similar Taylor-type expansion holds for $v_\O$ and $g$. 
Therefore, in light of this notation, by \cite[Lemma 2.6]{ButtazzoMaialeMazzoleniTortoneVelichkov:ARMA} we have that 
\begin{align}\label{e:first-derivative-F-along-vector-field}
\begin{aligned}
\delta J(\O)[\xi]=&\, \int_\O \left(\nabla u_\Omega \cdot \nabla v_\Omega + \Lambda \right)\dive\xi - \nabla u_\O \cdot ((\nabla\xi) + (D\xi))\nabla v_\O \, dx\\
&\,-\int_\Omega u_\O \dive(g\xi) + v_\O \dive(f\xi)\, dx.
\end{aligned}
\end{align}
Moreover, if $\partial\Omega$ is $C^2$-regular in a neighborhood of the support of $\xi$, then
\be\label{e:statiosmooth}
\delta J (\O)[\xi] = \int_{\partial \Omega} (\nu \cdot \xi)\big(\Lambda-|\nabla u_\O ||\nabla v_\O|\big)\, d\HH^{d-1},
\ee
where $\nu$ is the outer unit normal to $\partial \Omega$. Similarly, by \cite[Proposition 2.8]{ButtazzoMaialeMazzoleniTortoneVelichkov:ARMA}, we have
\be\label{e:second-derivative-F-along-vector-field}
\frac12 \delta^2 J(\O)[\xi]
=\,
\int_\O\Big(\nabla u_\Omega\cdot(\delta^2 A)\nabla v_\Omega-\nabla (\delta u_\O)\cdot \nabla (\delta v_\O)- (\delta^2 f) v_\Omega - (\delta^2 g) u_\Omega\Big) dx\,,
\ee
where $\delta^2 A,\delta^2 f$, $\delta^2 g$, $\delta u_\O$ and $\delta v_\O$ are the ones previously defined.
We stress again that in the formulations of the first and the second derivatives \eqref{e:first-derivative-F-along-vector-field}-\eqref{e:second-derivative-F-along-vector-field} it is not necessary to assume regularity of the optimal shape, indeed they are defined for every open set $\O$. The first and the second variation formulas can seem quite involved at this point, but both expressions will get simpler at blow-up. 

\subsubsection{The triple blow-up} The last obstruction is in the fact that the first order optimality condition (i.e., $\delta J(\O)[\xi]=0$ for every smooth vector field with compact support) is not leading to a monotonicity formula for $u_\Omega$ and $v_\Omega$; in particular, we do not know if the blow-up limits of $u_\Omega$ and $v_\Omega$ are in general homogeneous.\medskip

In order to solve this issue we analyze the blow-up sequences of stable critical points of the shape functional $J$ (i.e., shapes $\O$ satisfying \eqref{e:intro-stable-critical-point}). 
Indeed, in \cite{ButtazzoMaialeMazzoleniTortoneVelichkov:ARMA} we show that such notion is stable under blow-up limits at free boundary points 
$$u_0(x)=\lim_{n\to\infty}\frac{u_{\Omega}(x_0+r_nx)}{r_n}\quad\text{and}\quad v_0(x)=\lim_{n\to\infty}\frac{v_{\Omega}(x_0+r_nx)}{r_n}\,\qquad\mbox{for all}\quad x_0 \in \partial \O.$$
Thus $\Omega_0=\{u_0>0\}=\{v_0>0\}$ is a stable critical point for the same functional, but this time with $f\equiv g\equiv0$. Now, since $u_0,v_0:\R^d\to\R$ are harmonic on $\Omega_0$, we can apply the Boundary Harnack Principle from \cite{MaialeTortoneVelichkov:AnnSNS} to obtain that the ratio $u_0/v_0$ is H\"older continuous in $\Omega_0$, up to $\partial\Omega_0$. After a second blow-up (this time centered in the origin), we get that the positivity set $\Omega_{00}$ of the functions
$$u_{00}(x)=\lim_{m\to\infty}\frac{u_{0}(r_mx)}{r_m}\qquad\text{and}\qquad v_{00}(x)=\lim_{m\to\infty}\frac{v_{0}(r_mx)}{r_m}\,,$$
is a stable critical point (in every ball $B_R\subset\R^d$) for the functional $\mathcal F(\cdot,B_R)$, still with $f\equiv g\equiv0$. Moreover, by the H\"older continuity of $u_0/v_0$, we  have that $u_{00}$ and $v_{00}$ are proportional. Thus, $u_{00}$ is (up to a constant) a stable critical point (in the sense of \eqref{e:intro-stable-critical-point}) for the Alt-Caffarelli's one-phase functional
$$\mathcal G(u;B_R):=\int_{B_R}|\nabla u|^2\,dx+|\{u>0\}\cap B_R|\quad\text{in every ball}\quad B_R\subset \R^d\,,$$
so the third blow-up in zero
$$u_{000}(x)=\lim_{k\to\infty}\frac{u_{00}(r_kx)}{r_k}\,$$
is a 1-homogeneous stable critical point for the Alt-Caffarelli's functional $\mathcal G$ (see \cref{def:global-stable-solutions} for the precise notion of stability, which is different from the one introduced by Caffarelli, Jerison and Kenig in \cite{CaffarelliJerisonKenig04:NoSingularCones3D}). Therefore, by a diagonal argument, at every point on $\partial\Omega$ there exists of a blow-up sequence converging to some homogeneous stable critical point for the Alt-Caffarelli's functional.\medskip

The existence of a homogeneous blow-up $u_{000}$ is a key element in the local analysis of free boundaries. Indeed, it allows to prove that the state functions $u_\Omega$ and $v_\Omega$ satisfies the free boundary condition
\begin{equation}\label{e:free-boundary-condition-product}
|\nabla u_\Omega||\nabla v_\Omega|=\Lambda \quad \mbox{on }\partial \O,
\end{equation}
in a viscosity sense, by showing that if we take a free boundary point admitting a one-sided tangent ball, then the homogeneous blow-up limits $u_{000}$ and $v_{000}$ constructed above are half-plane solutions. This, in combination with the epsilon-regularity theorem from \cite{MaialeTortoneVelichkov:RMI}, implies that, in a neighborhood of any boundary point admitting a half-plane solution as blow-up limit, the free boundary is $C^{1,\alpha}$-regular (see Subsection \ref{s:eps-fg} for a precise discussion on the role of $\eps$-regularity).

\subsubsection{On the notion of stability for the one-phase problem}
We conclude the section by discussing the role of the stability in the study of singular points, namely those points in which the blow-up limit is not an half-plane solution. It is natural to expect that the stability of $\Omega$ (the second part of \eqref{e:intro-stable-critical-point}) leads to an estimate on the dimension of the singular set; in fact, the bounds on the critical dimension (the dimension in which a singularity appears for the first time) for minimizers of the Alt-Caffarelli functional rely (see \cite{CaffarelliJerisonKenig04:NoSingularCones3D} and \cite{JerisonSavin15:NoSingularCones4D}) on the well-known {\it stability inequality} of Caffarelli, Jerison and Kenig, which was originally obtained in \cite{CaffarelliJerisonKenig04:NoSingularCones3D} through a particular second order variation of the functional $\mathcal G$.

For the sake of completeness, we give here the notion of stability arising from the blow-up analysis.

\begin{definition}[Global stable solutions of the one-phase problem]\label{def:global-stable-solutions}
A function $u:\R^d\to\R$ is said to be a \emph{global stable solution of the Alt-Caffarelli problem} if, for every smooth vector field with compact support $\eta \in C^\infty_c(\R^d;\R^d)$, we have:
\be\label{e:global-stable-solutions-main}
\delta \mathcal G(u)[\eta]=0\qquad\text{and}\qquad \delta^2 \mathcal G(u)[\eta]\ge0\,,
\ee
and if the following conditions are fulfilled:
\begin{enumerate}
\item[$\qquad(\mathcal A1)$]\label{item:def-stable-1} $u$ is globally Lipschitz continuous and non-negative on $\R^d$\rm ;\it
\item[$\qquad(\mathcal A2)$]\label{item:def-stable-2} $u$ is harmonic in the open set $\O_u$\rm;\it
\item[$\qquad(\mathcal A3)$]\label{item:def-stable-3} there is a constant $c>0$ such that
$$|B_r(x_0)\cap \O_u| \leq (1-c)|B_r|,$$
for every $x_0\in \R^d\setminus\O_u$ and every $r>0$\rm ;\it
\item[$\qquad(\mathcal A4)$]\label{item:def-stable-4} $0\in\partial\O_u$ and there is a constant $C > 0$ such that
$$\sup_{B_r(x_0)} u \ge \frac{1}{C} r\quad\text{for every}\quad x_0\in \overline\O_{u}\,\text{ and }\,r>0\,;$$
\item[$\qquad(\mathcal A5)$]\label{item:def-stable-5} there is a constant $C>0$ such that, for every $R>0$,
$$\left|\int_{B_{R}}\nabla u\cdot\nabla\varphi\,dx\right|\le C R^{d-1}\|\varphi\|_{L^\infty(B_{R})}\quad\text{for every}\quad \varphi\in C^\infty_c(B_R)\,.$$
\end{enumerate}
\end{definition}
Although the notion of stability, combined with the existence of a one-homogeneous triple blow-up, allows one to estimate the dimension of the singular set in terms of the smallest dimension admitting one-homogeneous stable solutions with singularities, several difficulties arise. Ultimately we can prove that on smooth cones (that is, cones with isolated singularity) this notion of stability is equivalent to the stability inequality of Caffarelli, Jerison and Kenig. Thus, we obtain the bound $5\le d^\ast\le 7$ on the critical dimension $d^\ast$ as a consequence of the results of Jerison and Savin \cite{JerisonSavin15:NoSingularCones4D} and De Silva and Jerison \cite{DeSilvaJerison09:SingularConesIn7D}. Finally, by the Federer's dimension reduction principle (and the epsilon-regularity theory), we obtain the bound on the Hausdorff dimension of the singular set. 

\subsection{The $\eps$-regularity theory}\label{s:eps-fg}
By the previous section, we can define the regular part of free boundary $\mathrm{Reg}(\partial \O)$ as the collection of boundary points for which at least one blow-up limit is a half-plane solution. Precisely, we say that $x_0 \in \mathrm{Reg}(\partial \O)$ if there exists a unitary vector $\nu \in\R^d$, positive coefficients $\alpha,\beta \in \R$, and blow-up limits $u_0,v_0$ such that 
$$
u_0(x) = \alpha(x\cdot \nu)^+,\quad v_0(x) = \beta(x\cdot \nu)^+\quad\mbox{ with }\quad \alpha \beta=\Lambda.
$$
A distinctive phenomenon emerges in this context: at regular points, although the state variables resemble multiples of the distance function, the individual slopes of these functions are not explicitly known. Instead, we only have global information, specifically that the product of the slopes remains constant across the entire regular set. Therefore, the free boundary condition \eqref{e:free-boundary-condition-product} can be written as 
$$
|\nabla \sqrt{u_\O v_\O}| = \sqrt{\Lambda} \quad \text{on } \mathrm{Reg}(\partial \O).
$$
In light of the analysis in \cite{MaialeTortoneVelichkov:RMI}, we can start by considering the following viscosity formulation of the free boundary condition.

\begin{definition}\label{d:viscos-fg}
Let $u_\O, v_\O$ be the two state variables associated to $\O\subset D$. We say that
\be\label{e.Vfg}
|\nabla \sqrt{u_\O v_\O}| = \sqrt{\Lambda} \quad \text{on } \partial \Omega_u \cap D
\ee
holds in the viscosity sense if, for every $x_0 \in \partial \Omega_u \cap D$ and every test function $\varphi \in C^2(D)$ with $|\nabla \varphi(x_0)| \neq 0$, the following holds:
\begin{enumerate}
\item[(i)] if $\varphi^+$ touches $\sqrt{u_\O v_\O}$ from below at $x_0$, then $|\nabla \varphi(x_0)|\le 1$;
\item[(ii)] if $\varphi^+$ touches $\sqrt{u_\O v_\O}$ from above at $x_0$, then $|\nabla \varphi(x_0)|\ge 1$;
\item[(iii)] if $a,b \in \R$ are such that 
$$
a>0,\quad b>0\quad\mbox{and}\quad ab=1,
$$
then if $\varphi^+$ touches $w_{ab}:=\frac12(au_\O+bv_\O)$ from above at $x_0$, then $|\nabla \varphi(x_0)|\ge 1$. 
\end{enumerate}
\end{definition}
From the previous blow-up analysis, it follows immediately that minimizers satisfy the free boundary condition in the aforementioned sense. Moreover, it is immediate to observe that if $u_\O$ and $v_\O$ satisfy \cref{d:viscos-fg}, then so do $\alpha u_\O$ and $\beta v_\O$, provided that $\alpha\beta = 1$.
\begin{lemma}[Improvement of flatness] Let $u_\O,v_\O$ be $\eps$-flat viscosity solution to \eqref{e.Vfg} with $0 \in \partial \O_u$. Namely, up to a rotation in $\R^d$, assume that
\begin{equation*}
 |u_\O - \sqrt{\Lambda}(x_d)^+|\leq \varepsilon \qquad \text{and} \qquad  |v_\O - \sqrt{\Lambda}(x_d)^+|\leq \varepsilon \qquad\text{in } B_1.
\end{equation*}
Then, there exists $r_0,\eps_0>0$ universal such that, if $\eps\in (0,\eps_0]$
\begin{equation*}
  |u_\O - \alpha(x\cdot \nu)^+|\leq \frac{\varepsilon}{2}r_0 \quad \text{and} \quad  |v_\O - \beta(x\cdot \nu)^+|\leq \frac{\varepsilon}{2}r_0\qquad \text{in }B_{r_0},
\end{equation*}
where $\nu\in \R^d$ is a unit direction such that $|\nu- e_d|\leq C\eps$ and $\alpha, \beta$ are positive constants satisfying
\begin{equation*}
\alpha\cdot\beta=\Lambda,\qquad |\sqrt{\Lambda} - {\alpha}|\le C\varepsilon\qquad\text{and}\qquad |\sqrt{\Lambda}-{\beta}| \le C\varepsilon.
\end{equation*}
\end{lemma}
By the definition of flat solutions, it is evident that the "improvement" must be expressed both in terms of the new direction of the approximating flat domain (i.e., the vector $\nu \in \mathbb{R}^d$) and the coefficients associated with the free boundary condition (i.e., $\alpha, \beta \in \R$). By controlling how these variables evolve at smaller scales, one can deduce both the existence of a unique normal vector at $0 \in \partial \Omega_u$ and the values of the slopes $|\nabla u_\Omega(0)|$ and $|\nabla u_\Omega(0)|$. This phenomenon reflects the invariance of the free boundary conditions under scalar multiplication of the state variables.
 
The main idea of the proof consists in tracking the improvement of the variables $\sqrt{u_\O v_\O}$ and $\frac12(u_\O+v_\O)$, rather than working component-wise. For instance, when $f=g=0$ the equations for $\sqrt{u_\O v_\O}$ and $\frac12(u_\O+v_\O)$ assume the following simple form 
$$
\begin{cases}
\Delta \sqrt{u_\O v_\O} \geq 0 &\mbox{in }\O_u\cap B_1\\
|\nabla \sqrt{u_\O v_\O}|=\sqrt{\Lambda} &\mbox{on }\partial \O_u\cap B_1,
\end{cases}\quad
\begin{cases}
\Delta(\frac12(u_\O + v_\O))   = 0 &\mbox{in }\O_u\cap B_1\\
|\nabla(\frac12(u_\O + v_\O)|\leq \sqrt{\Lambda} &\mbox{on }\partial \O_u\cap B_1,
\end{cases}
$$
namely they are respectively a viscosity subsolution and supersolution of the (scalar) one-phase problem (in the previous computations we exploit that $u_\O, v_\O$ are non-negative). 
Now, it is easy to see that the if $u_\O, v_\O$ are $\eps$-flat, then the two auxiliary variables inherit the flatness assumption.\\
Notice that if $u_\O, v_\O$ are $\eps$-flat, then
$$
0\leq \frac{u_\O + v_\O}{2}- \sqrt{u_\O v_\O} \leq C\eps^2\qquad \mbox{in }\O_u\cap B_1.
$$
Thus, we can show that the following dichotomy holds:
\begin{enumerate}
    \item[(i)] either the flatness $\frac12(u_\O+v_\O)$ is improved from above
    \item[(ii)] or the flatness of $\sqrt{u_\O v_\O}$ is improved from below. 
\end{enumerate}
Now, notice that we cannot transfer this information back to $u_\O$ and $v_\O$ just be an algebraic manipulation; for instance, a bound from below on $\sqrt{u_\O v_\O}$ does not a priori imply a bound from below on both $u_\O$ and $v_\Omega$. On the other hand, one can easily notice that the improved flatness of $\sqrt{u_\O v_\O}$ or $\frac12(u_\O+v_\O)$, in particular, implies that in smaller (universal) ball the boundary $\partial\Omega$ is trapped between two nearby translations of a half-space. Using this geometric information and a comparison argument based on the boundary Harnack principle, we can deduce that also the flatness of $u_\O$ and $v_\O$ improves in a smaller (universal) scale.
\subsection{Analysis of the singular set under measure constrain}\label{s:singular.constrained} A remarkable difference between the penalized problem and the one with measure constrain lies in the study of singular points. Indeed, by following the program from Subsection \ref{s:triple}, it is necessary to formulate a notion of stability with respect to volume preserving inner variations.
Therefore, given any smooth vector field with compact support $\eta \in C^\infty_c(\R^d;\R^d)$, we need to construct a \emph{volume preserving flow}, whose infinitesimal generator coincides with $\eta$ in a neighborhood of $\{\eta\neq 0\}$. Then, by extending to the case of second variations the well-known program for measure-constrained optimization problems (see, for instance, \cite[Section 11]{Velichkov:BookRegularityOnePhaseFreeBd} and \cite[Section 17.5]{maggi}), we can show that if $\O$ is a shape optimizer for the measure constrained problem, then there exists a Lagrange multiplier $\Lambda \geq 0$ such that $\Omega$ is stable for the functional
    $$
J_\Lambda(\O) := J(\O)  + \Lambda |\O|,
$$
with respect to volume-preserving domain variations (see \cite[Theorem 2.1]{MazzoleniTortoneVelichkov:JConvexAnal}). Precisely \begin{equation}
\delta J_\Lambda(\O)[\eta]=0\qquad\text{and}\qquad \delta^2 J_\Lambda(\O)\left[\eta - \frac{\int_\O \dive\eta\,dx}{\int_\O \dive\xi\,dx}\xi\right]\ge0\,,
\end{equation}
for every smooth vector fields $\eta,\xi \in C^\infty_c(\R^d;\R^d)$ with disjoint compact supports satisfying
$$
\int_\O \dive\xi \,dx\neq 0.
$$
Then, by following the argument presented in Subsection \ref{s:triple}, it is possible to show that the notion of stability with respect to volume-preserving variations is preserved under blow-up. Precisely, given $x_0 \in \partial \O\cap D$, there exists a sequence $r_k\to 0$ such that  
$$\lim_{k\to\infty}\frac{1}{r_k}u_\O(x_0+r_k x)=u_0(x)\qquad\text{and}\qquad \lim_{k\to\infty}\frac{1}{r_k}v_\O(x_0+r_k x)=v_0(x)\,$$
locally uniformly in $\R^d$, strongly in $H^1_{loc}(\R^d)$. Moreover, $u_0$ and $v_0$ are proportional and there exists $\lambda>0$ such that $\lambda u_0$ is a global stable solution of the one-phase problem along mean-zero variations.
For the sake of completeness, we recall here the notion of stability with respect to mean-zero variations introduced in \cite{MazzoleniTortoneVelichkov:JConvexAnal}.
\begin{definition}[Global stability along mean-zero variations of the one-phase problem]\label{def:global-stable-solutions-meanzero}
A function $u:\R^d\to\R$ is said to be a global stable solution of Alt-Caffarelli problem \emph{along mean-zero variations} if:
\begin{enumerate}
\item[\rm(1)] conditions $(\mathcal A1), (\mathcal A2), (\mathcal A3)$, $(\mathcal A4)$, $(\mathcal A5)$ of \cref{def:global-stable-solutions} are satisfied\,\rm;\it
\item[\rm(2)] for every couple $\eta,\xi \in C^\infty_c(\R^d;\R^d)$ of smooth vector fields with disjoint compact supports satisfying
    $$
\int_{\O_u} \dive\xi \,dx\neq 0,
$$
it holds that
\be\label{e:stab.mean}
\delta \mathcal G_\Lambda(u)[\eta]=0\qquad\text{and}\qquad \delta^2 \mathcal G_\Lambda(u)\left[\eta - \frac{\int_{\O_u} \dive\eta\,dx}{\int_{\O_u} \dive\xi\,dx}\xi\right]\ge0\,.
\ee
\end{enumerate}
\end{definition}
Thus, the proof is concluded once we show that the estimates on the dimension of the singular part of the free boundary hold for solutions of the Alt-Caffarelli functional that are stable only with respect to mean-zero variations (see \cite[Theorem 1.9]{MazzoleniTortoneVelichkov:JConvexAnal}). The main idea behind this argument is to show that, for homogeneous solution, the stability with respect to mean-zero variations is indeed equivalent to the one defiend by Caffarelli, Jerison and Kenig in \cite{CaffarelliJerisonKenig04:NoSingularCones3D}.
\section{Open problems}
We conclude this survey by collecting some open questions related specifically to the vectorial Bernoulli free boundary problem. For the sake of completeness, we wish to highlight that, in recent years, many authors have contributed to the broader theory of vector-valued free boundary problems by exploring similar generalizations in the context of obstacle-type problems, semilinear equations, and constrained maps (e.g., \cite{jeon1,jeon2,FigalliGuerraKim2023,figalli2024,desilvaSavin:energyharmonic}). Therefore, in this final section, we restrict our focus to open problems that directly stem from the results and techniques presented in this survey.
\subsection{Optimal domains for functionals involving Dirichlet Eigenvalues}
In \cite{KriventsovLin2018:nondegenerate} the authors were interested in the differences between \vspace{0.1cm}
$$
\text{(I) $F$ is non-degenerate} \quad\longrightarrow\quad \text{(II)\, \txt{$F$ is differentiable \\ with respect to \\ domain variations.}}
$$
\phantom{c}\vspace{-0.1cm}\\
Indeed, while under (II) the free boundary
condition can be obtained explicitly in terms of the derivatives of $F$ with respect to domain variations, in case (I) the free
boundary condition is deduced by an approximation argument and the domain should be thought of as \emph{stationary} for the corresponding free boundary problem.\\ 
Moreover, for what it concerns the study of singular points, the difference between the cases (I) and (II) can be rephrased in terms of Property S and Property E (see \eqref{e:property-SE}). Then, in the case of degenerate functionals \cite{KriventsovLin2019:degenerate}, the authors deepen the study of the approximation in relation to the problem of singular points. The main limitation of their approximation scheme lies in the loss of information in the passage to the limit as $p \to \infty $. In particular, it would be interesting to develop new approximation methods for both non-degenerate functionals of type (I) and degenerate ones, with the aim of recovering, in the original problem, the existence of a vector of eigenfunctions that satisfies a rigidity condition at singular points. Such condition should ideally emerge as the limit of second-order variations of the approximating problems with respect to domain deformations.
\begin{open}[Improved estimate for singular points]\label{e:open-singular}
Is it possible to improve that Hausdorff estimate for (one-phase) singular points $\mathrm{Sing}_1(\partial \O)$ in the case of degenerate functionals (also in the case of non-degenerate functionals non-differentiable with respect to domain variations)? 
\end{open}
Similarly, the regularity of the two-phase singular set remains an open problem. Since this problem persists even at the level of minimizers, we postpone the discussion to \cref{e:open-branching}.\\

In recent years, different authors considered the problem of understand the behavior of the boundary of optimal shapes $\O$ close the boundary of the fixed container $\partial D$. In the case of spectral optimization problem such result has been studied only in the case of minimization associated to the first and the second Dirichlet eigenvalues (see for instance \cite{RussTreyVelichkov:CalcVarPDE,MazzoleniTreyVelichkov:AnnIHP}), where the authors exploited the regularity result for the \emph{obstacle} one-phase Bernoulli problem of \cite{ChangLaraSavin:BoundaryRegularityOnePhase}. Therefore, it is possible to prove that in the case of non-degenerate functionals, 
$$
\partial \O = \mathrm{Reg}(\partial \O)\cup \mathrm{Sing}(\partial \O),
$$
where $\mathrm{Reg}(\partial \O) = \mathrm{Reg}(\partial \O\cap D) \cup (\partial \O\cap \partial D)$ is locally the graph of a $C^{1,1/2}$-function. On the other hand, the singular set $ \mathrm{Sing}(\partial \O)$ is strictly contained in $D$ (i.e., $\mathrm{Sing}(\partial \O) =  \mathrm{Sing}(\partial \O\cap D)$) and the result coincides with the one of \cref{s:eigenvalues}. Thus, a natural question arises.
\begin{open}[Regularity close to Dirichlet boundaries]
    What can we say about the boundary regularity for optimal shapes associated to degenerate functionals? In other words, what is the behavior of the boundary of optimal shapes $\partial \O$ close to the Dirichlet fixed boundary $\partial D$?   
\end{open}
Notice also that the analysis of shape optimizers related to spectral problems has been recently generalized to the case of vector-valued shape optimization problems involving the spectrum of a non-local operator. In \cite{Tortone:shapefractional} the first named author considered the minimization problem associated to eigenvalues of the fractional Laplacian
$$
\begin{aligned}
\lambda_1^s(\Omega)&:=\min\left\{\frac{[u]^2_{H^s(\R^d)}}{\|u\|^2_{L^2(\R^d)}} : u\in H^s_0(\Omega)\setminus \{0\}\right\},\\
\lambda_i^s(\Omega)&:=\min_{E_i\subset H^s_0(\Omega)} \max_{u\in E_i\setminus \{0\}}\frac{[u]^2_{H^s(\R^d)}}{\|u\|^2_{L^2(\R^d)}},\quad\text{if $i\geq 2$},
\end{aligned}
$$
where the latter minimum is taken over all $i$-dimensional linear subspaces $E_i$ of $H^s_0(\Omega)$. Among several differences, in this case the free boundary condition takes the form
$$
\lim_{t \to 0^+}\frac{|U|(x+t\nu_x)}{t^s} = \frac{1}{\Gamma(1+s)}\frac{2s}{d}\sum_{i=1}^k \lambda_i^s(\O) \qquad \mbox{for every }x \in \partial\O\cap D,
$$
where $\nu_x$ is the normal vector at $x\in \partial \O$ pointing towards the interior of $\O$ and $|U|$ the vector of the first $k$ normalized eigenfunctions. Precisely, in \cite{Tortone:shapefractional} has been studied the case of non-degenerate vectorial problems
$$
\mathrm{min}\left\{\lambda_1^s(\O)+\cdots+\lambda_k^s(\O)+\Lambda|\O|\colon \O\subset D\text{ open}\right\}
$$
for some bounded smooth container $D\subset \R^d$. A first result concerning the study of degenerate problems has recently been obtained: specifically, in \cite{zahl}, the shape optimization problem associated with the $k$-th eigenvalue of the $s$-Laplacian was investigated under the assumption of simplicity. As can be seen, the picture is far from being complete, and it would be interesting to understand whether the nonlocal effect of the operator influences both the optimal shapes and their regularity.
\begin{open}[Degenerate function of fractional eigenvalues]
Given $s \in (0,1)$, can the analysis in \cite{Tortone:shapefractional} be extended to the class of degenerate functionals $F \colon \mathbb{R}^k \to \mathbb{R}$, without assuming simplicity of the $k$-th eigenvalue?
\end{open}
Moreover, it will be interesting to adapt the methodologies recently introduced in \cite{RosOtonWeidner:optimal,RosOtonWeidner:improvement} to the spectral case. 
\begin{open}[Nonlocal spectral problem]
Given $s\in (0,1)$, consider the class of symmetric $2s$-stable integro-differential operators defined by 
$$
Lu(x):= 2\,\mathrm{p.v. }\int_{\R^n}(u(x)-u(y))K(x-y)\,dy,
$$
where $K\colon \R^d \to [0,+\infty]$ is assumed to satisfy
$$
\lambda |h|^{-d-2s}\leq K(h)\leq \Lambda |h|^{-d-2s},\quad
K(h)=K(-h),\quad K(h)=\frac{K(h/|h|)}{|y|^{d+2s}},
$$
for some $0<\lambda\leq \Lambda $. Can the analysis carried out for the spectral optimization problem associated with $(-\Delta)^s$ be extended to this class of integro-differential operators? 
\end{open}

\subsection{Vectorial Bernoulli problem}  For what concerns the two-phase singular points of rank higher than one, the classification of matrices which are global minimizers is currently an open problem. Moreover, the structure of the two-phase singular set closely resembles that of the singular set in the classical obstacle problem. Nevertheless, the presence of a Bernoulli-type free boundary condition prevents a direct application of standard techniques such as monotonicity formulas (e.g., of Weiss, Almgren, or Monneau type), and thus hinders the development of higher-order blow-up analysis.
\begin{open}[Regularity of the two-phase singular set $\rm Sing_2$]\label{e:open-branching}
Classify the blow-up limits at points $x_0 \in \mathrm{Sing}_2(\partial \Omega_U)$ with rank $1 < \mathrm{Rk}(x_0) \leq k$. Moreover, for each $j = 1, \dots, k$, does there exist a modulus of continuity $\omega$ such that $\mathcal{S}_j$ is contained in a countable union of $j$-codimensional $C^{1,\omega}$ manifolds?\end{open}
Similarly, it could be interesting to understand the generic regularity for the vectorial one-phase problem: in light of the analogue result for the scalar setting \cite{generic-hui}, we may expect that 
$$
\text{generically $\mathrm{dim}_{\mathcal{H}}(\mathrm{Sing}_1(\partial \O_U)) \leq d-d^*-1.$}
$$
\begin{open}[Generic regularity]
Establish generic regularity of free boundaries for the vectorial Bernoulli problem.
\end{open}
We notice that we cannot expect that the two-phase singular set $\rm Sing_2$ becomes small for generic perturbations, since this is not true even for the scalar two-phase problem. 

\subsection{Optimal domains for Dirichlet energy functionals}
The analysis related to the \emph{triple consecutive blow-up} is quite robust and relies on mild assumptions, which suggest that it could be applied to more general choices of $f$ and $g$. Nevertheless, the main restriction of our analysis is that we consider $f,g \colon D \to \mathbb{R}$ satisfying the conditions
\begin{enumerate}
    \item[(a)] $f,g \in L^\infty(D)\cap C^\infty(D)$ and non-negative;
    \item[(b)] there are constants $C_1, C_2 > 0$ such that
    $
    0\leq C_1 g \leq f \leq C_2 g \mbox{ in } {D}.
    $
\end{enumerate}
As we have already emphasized, the second assumption is crucial for showing that the state variables of the problem are both non-degenerate and Lipschitz continuous, which is the natural starting point for developing the regularity theory of the free boundaries. It would be interesting to remove this assumption by considering
\begin{open}[Generalized $f$ and $g$]\label{o:generalizedfg}
Generalize the study of the state variables to the case of non-comparable $f$ and $g$.
\end{open}
Among various examples, particularly interesting is the problem for $j(u_A, x) = -u_A^2$, where $u_A$ is the torsion function (i.e., $-\Delta u_A = 1 $ in $A$). This corresponds to the case $f = 1$ and $g = u_A$.

\section*{Acknowledgements}
The authors were supported by the European Research Council's (ERC) project n.853404 ERC VaReg - \it Variational approach to the regularity of the free boundaries\rm, financed by the program Horizon 2020. The authors acknowledge the MIUR Excellence Department Project awarded to the Department of Mathematics, University of Pisa, CUP I57G22000700001. B.V. also acknowledges support from the project MUR-PRIN “NO3” (n.2022R537CS). G.T. is member of INdAM-GNAMPA.
\section*{Statements and Declarations} 
The authors declare that they have no conflict of interest regarding the publication of this manuscript.
\bibliographystyle{abbrv}
\bibliography{biblio.bib}

\end{document}